\documentclass[11pt]{article}

\usepackage{mathrsfs}

\usepackage[margin=1in]{geometry}
\usepackage[utf8]{inputenc}
\usepackage{enumerate}
\usepackage{amsfonts}
\usepackage{amsmath}
\usepackage{amssymb}
\usepackage{hyperref}
\usepackage{bbm}
\usepackage{amsthm}
\usepackage{xcolor}
\usepackage{tikz}
\usepackage{comment}

\usepackage{colonequals}
\usepackage{theoremref}
\usepackage{enumitem}
\numberwithin{equation}{section}
\usepackage{appendix}

\newcommand{\e}{\mathbb E}

\newcommand{\p}{\mathbb P}
\newcommand{\1}{\mathbbm 1}
\newcommand{\eps}{\varepsilon}

\newcommand{\la}{\langle}
\newcommand{\ra}{\rangle}
\newcommand{\odd}{\mathrm{odd}}

\newcommand{\Z}{\mathbb{Z}}

\theoremstyle{definition}

\newtheorem{theorem}{Theorem}[section]

\newtheorem{proposition}[theorem]{Proposition}
\newtheorem{lemma}[theorem]{Lemma}
\newtheorem{remark}[theorem]{Remark}

\newtheorem{definition}[theorem]{Definition}
\newtheorem{assumption}[theorem]{Assumption}

\begin{document}
\title{Disorder Chaos in Short-Range, Diluted,
	and L\'evy Spin Glasses}

 \author{Wei-Kuo Chen\thanks{Email: wkchen@umn.edu. Partly supported by NSF grants DMS-1752184 and DMS-2246715 and Simons Foundation grant 1027727 } \and  Heejune Kim \thanks{ Email: kim01154@umn.edu. Partly supported by NSF grants DMS-1752184 and DMS-2246715}  \and Arnab Sen \thanks{Email: arnab@umn.edu.Partly supported by Simon Foundation MP-TSM-00002716}} 

\maketitle 

\begin{abstract}
 In a recent breakthrough \cite{Chatterjee}, Chatterjee proved site disorder chaos in the  Edwards-Anderson (EA) short-range spin glass model \cite{Edwards1975} utilizing the Hermite spectral method. In this paper, we demonstrate the further usefulness of this Hermite spectral approach by extending the validity of site disorder chaos in three related spin glass models.
	The first, called the mixed even $p$-spin short-range model, is a generalization of the EA model where the underlying graph is a deterministic bounded degree hypergraph consisting of hyperedges with even number of vertices. The second model is the diluted mixed $p$-spin model, which is allowed to have hyperedges with both odd and even number of vertices. For both models, our results hold under general symmetric disorder distributions. 
The main novelty of our argument is played by an elementary algebraic equation for the Fourier-Hermite series coefficients for the two-spin correlation functions. It allows us to deduce necessary geometric conditions to determine the contributing coefficients in the overlap function, which in spirit is the same as the crucial Lemma 1 in \cite{Chatterjee}. Finally, we also establish 
disorder chaos in the L\'evy model with stable index $\alpha \in (1, 2)$. 
	
\end{abstract}



\section{Introduction and Main Results}



The study of chaos in spin glasses began from \cite{BM87, FH86}, which is concerned with the sensitivity of the systems subject to a small change in some external parameters, such as the temperature, disorder, or external field.
One typical way to capture chaotic phenomenon is to measure the site overlap between two configurations independently drawn from an unperturbed and a perturbed system, respectively. The expectation is that with the perturbation, the site overlap is concentrated near zero if the model is symmetric with respect to the sign change of the spin configuration and is concentrated near a non-zero value if the symmetry is broken, for example, due to the presence of the external field.
This is in sharp contrast to the ``lack of self-averaging'' behavior of the site overlap in the low temperature phase of many models if the two spin configurations are independently sampled from the same system (see, e.g., Example 1 in \cite{Panchenko08}). For a modern survey of chaos in spin glasses in the physics literature, see \cite{Rizzo09}.

Establishing disorder chaos in the Edwards-Anderson (EA) short-range spin glass model \cite{Edwards1975} has received great attention recently. 
In \cite{chatterjee2008chaos}, Chatterjee established a weak form of disorder chaos for the site overlap in the sense that it is concentrated around its Gibbs average; an alternative proof was given in Chen-Panchenko \cite{CP14}.  In addition to the site overlap, which aims to capture chaos through spin interactions on vertices,  the bond overlap that measures chaos through the spin interactions on edges is another quantity of interest and it was shown in \cite{chatterjee2008chaos} that the expectation of the bond overlap is strictly positive.  In a related direction, Arguin-Newman-Stein \cite{ANS} made a connection between bond disorder chaos and incongruent ground states, and Arguin-Hanson \cite{AH} reproved positivity of the expectation of the bond overlap in \cite{chatterjee2008chaos} in three different perspectives.

In a recent breakthrough, chaos in disorder for the site overlap at zero temperature in the EA model was settled by Chatterjee \cite{Chatterjee} through a beautiful adaptation of the Hermite spectral method. The main ingredient of his proof is a quantitative version, \cite[Theorem 2]{Chatterjee}, of a statement regarding an expected glassy behavior: ``relative orientations of spins with large separation are sensitive to small changes in the bond strengths'',  quoted from Bray-Moore \cite{BM87}. Following the validity of the site disorder chaos, a number of long-standing conjectured spin glass behaviors were answered quantitatively.

In the present paper, we aim to investigate chaos in disorder in three spin glass models. The first, called the mixed $p$-spin short-range model, is a generalization of the EA model, whose Hamiltonian is defined on a mixture of hypergraphs of varying lengths $p$ and the spin interactions are placed along the hyperedges. The second model called the diluted mixed $p$-spin model is a randomization of the first model, where each $p$-hyperedge is chosen independently with the same success probability. We assume that the disorder distribution is symmetric around zero. Therefore, it can always be written as $\rho(J)$ for an odd measurable function $\rho$ and $J\sim N(0,1)$. For any $t>0,$ we perturb the disorder in the original systems in continuous and discrete manners, namely, $$\rho(e^{-t}J+\sqrt{1-e^{-2t}}J')$$ and $$B\rho(J)+(1-B)\rho(J')=\rho(BJ+(1-B)J'),$$ where $J'$ is an independent copy of $J$ and $B$ is Bernoulli with head probability $e^{-t}$ independent of $J,J'.$  Under these two perturbations, we establish chaos in disorder for the site overlap in the mixed {\it even} $p$-spin short-range model that generalizes the earlier result  \cite{Chatterjee} in the EA model. Additionally, we show that the same holds in the diluted mixed $p$-spin model, 
which was one of the major open problems in the class of mean-field models and to the best of our knowledge, it was previously known only for the diluted even $p$-spin model with large connectivity, see Chen-Panchenko \cite{CP_2018}. Our main result in this direction holds without this constraint and most importantly, it works for any mixture including odd $p$. 

Our approach to establishing site disorder chaos in the two aforementioned models is based on the Hermite spectral method; the key ingredient is played by an elementary algebraic equation stated in Lemma \ref{lem: coefficient zero} for the Fourier-Hermite series coefficients associated with the two-point spin correlation function, which allows us to deduce certain geometric conditions, see Propositions \ref{prop: lowfrequencydies} and \ref{prop: hypertree lemma}, to identify the nonzero (contributing) Fourier coefficients. These propositions are similar in spirit to \cite[Lemma 1]{Chatterjee}. It was derived in the EA model through flipping the signs of the spins as well as the disorders in the Hamiltonian depending on the graph structure associated with the given two vertices heavily relying on the fact that the EA model in the absence of an external field is symmetric. However, our algebraic Lemma~\ref{lem: coefficient zero} does not depend too much on the geometry of the underlying graph. As such, it substantially broadens the scope of applicability of the Hermite spectral method in deducing disorder chaos in a wider range of spin glass models including diluted mixed $p$-spin models.

The third model of interest is the L\'evy spin glass initially introduced by Cizeau-Bouchaud \cite{Cizeau1993,Cizeau1994}. Its Hamiltonian is essentially the same as the Sherrington-Kirkpatrick (SK) model except that its disorder possesses a power-law density of order $1<\alpha<2$ so that it only has $\alpha'$-moment for any $\alpha'<\alpha.$ Due to this formulation, it is expected that only the edges with heavy weights will be responsible for the behavior of the system and as a result, the L\'evy model should behave like diluted models in many aspects, see \cite{Cizeau1993} for physicists' predictions. Indeed, it is the case that one can relate the L\'evy model to the diluted $2$-spin model to understand its free energy and fluctuation, see Chen-Kim-Sen \cite{CKS}. It is then an intriguing question to ask whether chaos in disorder holds in the L\'evy model under the same mechanism as the diluted model or not. To this end,  we let the disorder be an arbitrary symmetric heavy-tailed random variable and write the disorder as $\rho(J)$ for some odd function $\rho$ and again $J\sim N(0,1)$ such that $\rho(J)$ belongs to the domain of attraction of the symmetric stable distribution of index $\alpha$. Under the continuous perturbation as before, we show that the L\'evy model exhibits disorder chaos. Somewhat unexpectedly, while one can understand the free energy in the L\'evy model via the diluted model, it is by no means clear if one can use our chaos result for the diluted model to understand the same objective in the L\'evy model as the Gibbs measures are much more sensitive subject to small perturbation. As we shall see, our proof makes use of the Hermite spectral method again, but instead of counting the nonzero Fourier-Hermite coefficients under the same geometric conditions as before, it is in fact more useful to control the overlap through the treatment in the SK model, see \cite{Chatterjeebook,chen2019order}.

\subsection{Mixed Even $p$-spin short-range Model}
Let $\Delta \ge 2$. For $2\leq p\leq \Delta,$ denote by ${[N]\choose p}$ the collection of all choices of $p$ distinct elements from $[N]$. Let $G=([N],E)$ be a hypergraph with a mixture of hyperedges \begin{equation}\label{mixed edge set}
     E=\bigcup_{2\le p \le \Delta}E_p,
\end{equation} for some $E_p \subseteq {[N]\choose p}$.
An ordinary graph would correspond to $\Delta=2$, which is the case of 
 the EA model.
A \textit{ball of radius $r\in \mathbb Z_+$ centered at $v\in [N]$} is defined in the usual way as $B_r(v)= \{v'\in[N]: d(v,v') \le r \}$, where $d(v,v')$ is the hypergraph distance between the vertices $v$ and $v'$ defined in Section \ref{sec: prelim hypergraphs}.
We assume one of the following growth assumptions on $G$.
\begin{assumption}[Volume growth assumptions on $G$]\label{edge set assumption: mixed p spin}
\indent
\begin{enumerate}[label=(\roman*)]
   \item (Polynomial growth)  For some positive constants $A$ and $\theta$ independent of $N$, 
    \[ \max_{i\in [N]} |B_r(i)|\le A r^\theta, \ \ r\in \mathbb N.\]
    \item (Exponential growth) For some  constants $A$ and $\gamma > 1$ independent of $N$, 
    \[ \max_{i\in [N]} |B_r(i)|\le A\gamma^{ r}, \ \ r\in \mathbb N.\]
\end{enumerate}
\end{assumption}

For each hyperedge $e\in E$, let $J_e$ be an i.i.d.\ copy from the standard normal distribution.
For each $2\le p\le \Delta$,
let $\rho_p:\mathbb R\to\mathbb R$ be an odd function.
With the notation \begin{equation}\label{eq: product notation}
    \sigma_e=\prod_{i\in V(e)} \sigma_i, \quad \sigma\in\{\pm 1\}^N,
\end{equation} we consider the Hamiltonian \begin{equation}\label{hamiltonian}
    H_J(\sigma)=\sum_{2\le p \le \Delta}H_J^{(p)}(\sigma), \quad \sigma \in \{\pm 1\}^N,
\end{equation} where \[H_J^{(p)}(\sigma):=\sum_{e\in E_p}\rho_p(J_e)\sigma_e\] for $2\le p \le \Delta$.
Note that we can represent any symmetric disorder as $\rho(J)$.
The \textit{Gibbs measure} associated to (any) Hamiltonian $H_J$ at inverse temperature $\beta>0$ is given by \begin{equation}\label{Gibbs measure}
    G_{J,\beta}(\sigma)=\frac{e^{\beta H_J(\sigma)}}{Z_N}, \quad \sigma \in \{\pm 1\}^N,
\end{equation} where $Z_N=Z_N(J,\beta)=\sum_{\sigma \in \{\pm 1\}^N}\exp(\beta H_J(\sigma))$ is the normalizing factor called the  \textit{partition function}. Denote by $(\sigma^l)_{l\geq 1}$ i.i.d. samples (replicas) from $G_{J,\beta}^\infty$ and the expectation with respect to these randomness by $\la \cdot\ra_{J,\beta}.$
Note that, when $\beta=\infty$, $G_{J,\beta}$ is a uniform measure supported on the set of \textit{ground states}, which are spin configurations achieving the maximum value of $H_J$ given the disorder $J$. In this case, we shall understand $(\sigma^l)_{l\geq 1}$ as i.i.d. copies that are sampled uniformly at random from the set of ground states. 
We often omit the notation for $J$ and $\beta$ to enhance brevity.
Given two spin configurations $\sigma^1, \sigma^2 \in \{\pm 1\}^N$, we define their overlap to be \[R(\sigma^1, \sigma^2)= \frac{1}{N} \sum_{1\le i\le N}\sigma^1_i \sigma^2_i.\] 
Throughout the article, we consider either one of the following perturbations of the disorder $\rho(J)$.
\begin{definition}[Two types of perturbations]
 \label{def: perturbation}
    \indent
    \begin{enumerate}[label=(\roman*)]
        \item For $t\ge 0$, $J(t)= e^{-t}J+\sqrt{1-e^{-2t}}J'$, where $J'$ is an indepdent copy of $J$.
    
        \item For $t\ge 0$, $J(t)= BJ+(1-B)J'$ where $B\sim \mathrm{Ber}(e^{-t})$, $J'\sim J$, and all the random variables are independent.

        \end{enumerate}
        We call the perturbation in (i) as a \textit{continuous perturbation} and the one in (ii) as a \textit{discrete perturbation}.
\end{definition}
In both types of perturbation, larger $t\ge 0$ corresponds to more perturbation.
The terminology is obvious because, in the first type, we are continuously deforming the disorder, which contrasts with the discrete replacement of the disorder with an i.i.d. copy in the second type.
The continuous perturbation is natural when considering Gaussian disorders in view of the Ornstein-Uhlenbeck (OU) process, see \eqref{eq: semigroup: continuous}. 
The discrete one generally applies to any disorders, and it is the one often considered in the noise-sensitivity literature (see, e.g., \cite{garban2015noise}). We also comment that we may naturally associate the discrete perturbation with a Markov process called the independent flip process, a term coined by Chatterjee, see \cite[Chapter 7]{Chatterjeebook}.


As mentioned in the introduction, our focus lies in understanding the sensitivity of our models subject to the perturbation under Definition \ref{def: perturbation}. 
To make this notion precise, we consider the overlap $R(\sigma,\tau)$ where  $(\sigma,\tau)$ is sampled from the product measure $G_J \otimes G_{J(t)}$.
We always sample $(\sigma^l,\tau^l)_{l\ge 1}$ from the product measure $(G_J\otimes G_{J(t)})^{\otimes \infty}$.
The Gibbs average in this context will be denoted by $\la \cdot \ra_t$ to emphasize that $\tau$ is drawn from the perturbed Gibbs measure $G_{J(t)}$.
Note that these are only defined for the positive temperature case ($\beta<\infty$), and we understand the zero temperature case ($\beta=\infty$) as $\sigma$ and $\tau$ being independently sampled uniformly at random from the set of ground states of the Hamiltonian $H_J$ and $H_{J(t)}$, respectively. 

Our first main result focuses on the mixed even $p$-spin short-range model, i.e., all $p$'s are restricted to even values in \eqref{mixed edge set}.

\begin{theorem}\label{thm: mixed p spin disorder chaos} Consider the mixed even $p$-spin short-range model with the Hamiltonian \eqref{hamiltonian}. Let $t \ge 0$ and  $\beta \in [0, \infty]$. Then, for both types of perturbations and for any $r \in \mathbb Z_+$, we have 
\begin{equation}\label{eq:general_bd}
\e \la R(\sigma,\tau)^2\ra_t \le \frac{1}{N}\max_{i\in [N]}|B_{r}(i)| + e^{-tr}. 
\end{equation}
In particular, we can specialize into the following cases.
\begin{enumerate}
  \item Under Assumption \ref{edge set assumption: mixed p spin}(i), for any $t >0$ and  $\beta \in [0, \infty]$,
\begin{equation*}
    \e \la R(\sigma,\tau)^2\ra_t \le \frac{1}{N} + \frac{C}{N t^\theta}
\end{equation*}
for some constant $C>0$ depending on $A$ and $\theta$.

\item Under Assumption \ref{edge set assumption: mixed p spin}(ii), for any $t >0$ and $\beta \in [0, \infty]$,
\begin{equation} \label{eq: result for condition (ii)}
     \e \la R(\sigma,\tau)^2\ra_t  \le \frac{C}{N^{t/(t+\log \gamma)}} 
\end{equation}for some constant $C>0$ depending  on $A$ and $\gamma$.
\end{enumerate}
\end{theorem}
\begin{remark}\label{rem:Theorem_1.3}
\begin{enumerate}
    \item It is an easy fact (see, e.g., Lemma~\ref{lem: semigroup: main}) that $t \mapsto \e \la R(\sigma,\tau)^2\ra_t$ is non-increasing. Hence, the overlap bounds given in 
    Theorem~\ref{thm: mixed p spin disorder chaos} are more interesting for small $t$, that is, when the amount to perturbations to disorders is small. In the polynomial growth case, our result says that $ \e \la R(\sigma,\tau)^2\ra_t \to 0$ as soon as $t \gg N^{-1/\theta}$ (independent of $\beta)$. For $\beta = \infty$, this matches with the result obtained in \cite{Chatterjee} for the ground state of the Edwards-Anderson model on finite boxes of $\mathbb{Z}^d$. On the other hand, for the exponential growth case, our result implies that $ \e \la R(\sigma,\tau)^2\ra_t \to 0$ as soon as $t \gg (\log N)^{-1},$ again independent of $\beta$.

    \item  It is possible to establish  crude lower bounds on the overlap when $t$ is small enough. 
    Under the discrete perturbation, Equation \eqref{eq: semigroup: discrete} implies  $\e \la R(\sigma,\tau)^2\ra_t \ge c \e\la  R(\sigma^1,\sigma^2)^2 \ra$ for some constant $c>0$, independent of $\beta\in [0,\infty]$, whenever $t\le |E|^{-1}.$ 
    This holds without any restrictions on the hyperedge set $E$, and, in particular, for dense graphs as well. 
    Under the continuous perturbation, we can show that $t\le |E|^{-5/2-\eps}$ for some $\eps>0$ guarantees a similar lower bound with further regularity assumptions on the functions $(\rho_p)_{2\le p \le \Delta}$, though this bound appears to be non-optimal.
    Specializing into the canonical case of Gaussian disorders, we can find an alternative bound  $t\le   12^{-2}\beta^{-1}|E|^{-3/2} (\e \la R(\sigma^1,\sigma^2)^2\ra)^2$ which has a better exponent in $|E|$ but depends on $\beta$, implies a similar lower bound on the perturbed overlap.
    We defer the precise statements and proofs of these inequalities to the appendix.

\end{enumerate}



\end{remark}

\subsection{Diluted Mixed $p$-spin Model}\label{sec: diluted mixed p}


 Fix $\Delta \ge 2$. For each $2\leq p\leq \Delta$, each hyperedge from ${[N]\choose p}$ is selected independently with probability
\begin{equation*}
  \frac{\alpha_p N}{ {N \choose p} }
    \end{equation*}
for some $\alpha_p>0$, and let $E_p$ be the collection of chosen hyperedges. Consider the random hypergraph  $([N], E)$ for $E:=\bigcup_{2\le p \le \Delta } E_p$. 
    The Hamiltonian is the same as before, namely \eqref{hamiltonian}, where the disorder is independent of the randomness of the edge set.
While we considered deterministic graphs in the mixed even $p$-spin short-range model described in Theorem \ref{thm: mixed p spin disorder chaos}, the randomness of the graphs in the diluted mixed $p$-spin model aids in removing the restriction on the edges being even, as will be apparent in Proposition \ref{prop: hypertree lemma} and Lemma \ref{lem: nbd hypertree} below.

Let us fix the notation  \[\lambda \colonequals  \sum_{2\le p \le \Delta} p(p-1) \alpha_p.\]

\begin{theorem}\label{thm: diluted mixed p spin}


Consider the diluted mixed $p$-spin model. Let $t\ge 0$ and $\beta \in [0,\infty]$. Assume $\lambda > 1$. For both types of perturbations, there exists a constant $C>0$ depending only on $(\alpha_p)_{2\le p \le \Delta}$ such that \[\e \la R(\sigma,\tau)^2\ra_t \le \frac{C}{N^{t/(t+ 2\log \lambda )}}.\]
\end{theorem}
\begin{remark}\label{rmk1}
\begin{enumerate}

    \item Site disorder chaos for a version of the diluted model was studied earlier in Chen-Panchenko \cite{CP_2018}, where the result was limited only to the pure even $p$ spin model with very large $\alpha_p$ so that it is close to the mean-field model. Our result here holds without these restrictions.

    \item    The parameter $\lambda$ is the average vertex degree of the underlying random hypergraph. It is known  \cite{schmidt1985component} that the critical threshold is $\lambda=1$. The case $\lambda > 1$ corresponds to the supercritical regime when there is a giant (linear-sized) component with high probability. On the other hand, if $\lambda < 1$, the hypergraph is subcritical, that is,  all connected components have at most logarithmic size with high probability, which renders the diluted model less interesting in this regime.


    \item For the diluted pure $2$-spin model with Gaussian disorder, for any $\beta>0$, as $N\to\infty$, the unperturbed overlap $R(\sigma^1,\sigma^2)$ has a trivial distribution at zero in the annealed region $\alpha_2 \le 1/2$ (see \cite{Guerra2004}). On the other hand, if $\alpha_2 > 1/2$, it is expected that there exists some critical temperature such that below the criticality, the limiting overlap has a non-trivial distribution. In sharp contrast to this nontriviality, the result in Theorem \ref{thm: diluted mixed p spin} indicates that the overlap is trivialized at any temperature as long as the perturbation is in force. 

    \item We can view $\lambda$ as the bound on the (average) growth rate for the underlying random hypergraph of the diluted model (see Lemma~\ref{lem: expectation of infected}). However, let us emphasize that  Theorem~\ref{thm: diluted mixed p spin} does not assume that  hyperedges contain an even number of vertices. We circumvent the issue of the odd hyperedges by exploiting the fact that random hypergraphs are locally tree-like. However, in the implementation of the argument, we lose a multiplicative factor of $2$ in front of $\log \lambda$ compared to the bound \eqref{eq: result for condition (ii)}  obtained in the deterministic case with no odd hyperedges.

\end{enumerate}
  
\end{remark}

\subsection{L\'evy Model}\label{sec: levy model}
Let $1<\alpha<2.$ 
Let $ X$ be a symmetric random variable with tail probability
\[\p(| X|>x)= \frac{L(x)}{x^\alpha},\ \forall x>0,
\]
     where $L:(0,\infty)\to (0,\infty)$ is a slowly varying function at $\infty$.
It is well-known that for any $p<\alpha$, \begin{equation}\label{eq: integrability of heavy tails}
    \e |X|^p <\infty.
\end{equation}
 Potter's bound says that for any $\eta>0$, there exists $y_0>0$ such that for all $y\ge y_0$ and $s\ge 1$, \begin{align}
    (1-\eta)s^{-\eta}<\frac{L(sy)}{L(y)} < (1+\eta)s^{\eta}.\label{potter}
\end{align} In particular, \eqref{potter} dictates the behavior of a slowly varying function as $s\to\infty$.
Suppose that  $\rho:\mathbb{R}\to\mathbb{R}$ is a non-decreasing function such that
	$\rho(J)$ has the same distribution as $X$ for $J\sim N(0,1).$ 
Let $(J_{ij})_{1\leq i,j\leq N}$ be i.i.d. copies of $J.$
 Define the Hamiltonian of the L\'evy model as
	\[	H_N(\sigma)=\frac{1}{a_N}\sum_{i,j=1}^N\rho (J_{ij})\sigma_i\sigma_j, \quad \sigma \in \{\pm 1\}^N,	\]
    where $a_N=\inf\{x>0: \p(|X|> x)\le 1/N\}.$ 
    It is a standard result, e.g., \cite{BCC11}, in the study of heavy-tailed random variables that \begin{align}
    a_N=N^{1/\alpha}\ell (N)\label{eq: a_N}
\end{align} for some slowly varying function $\ell$ at $\infty$.

    The Gibbs measure for the L\'evy model is defined in the same way as \eqref{Gibbs measure} with the above Hamiltonian.

\begin{theorem}\label{chaosLevy}
Consider the L\'evy model with $1<\alpha<2$ and the continuous perturbation. 
 For any $0\leq \beta <\infty$ and $0<\eps< 2/\alpha -1$, there exist two constants $c=c(\beta)>0$ and $K=K(\beta,\eps)>0$ such that for any $t \ge 0$ and $N\geq 1,$ 
		\begin{equation*}
			\e\la R(\sigma,\tau)^2\ra_t \leq \frac{K}{N^{(\frac{2}{\alpha}-1-\eps)\min(1,ct)}}.
		\end{equation*}
	\end{theorem}
\begin{remark}
   As in the third item of Remark \ref{rmk1}, the nearly zero behavior of the overlap in Theorem \ref{chaosLevy} is a sharp transition as opposed to the unperturbed overlap $R(\sigma^1,\sigma^2)$ in the L\'evy model, for which it is expected to have a nontrivial distribution as long as $\beta$ is large enough, see 
   \cite{Cizeau1993,Cizeau1994}.
\end{remark}







\subsection{Related Results and Applications}

In addition to the short-range and graphical models, there has been vast mathematical progress on chaos in mean-field models in recent years. First of all, chaos in disorder in the SK model as well as its variant, the mixed even $p$-spin model, was pioneered by Chatterjee  \cite{Chatterjee2009}  (among many other results) in the absence of external field. Later on,  disorder chaos was extended by Chen \cite{Chen_2013,Chen16} in the presence of the external field and was pushed forward to the zero temperature setting by Chen-Handschy-Lerman \cite{CHL17}.
Some efforts (e.g. \cite{Chen16,CP_2012}) were made to go beyond the mixed even $p$-spin model and allow odd interactions, which culminated in Eldan's \cite{Eld20} proof in the general mixed $p$-spin model.
On the other hand, for the same class of models, if it was assumed that the mixture of the $p$-spin Hamiltonians is {\it generic}, a form of chaos in temperature was achieved by Chen-Panchenko \cite{CP_2012} and Chen \cite{Chen_2014} and the canonical case was settled by Panchenko \cite{Panchenko_2016}. Interestingly, in the spherical setting, while chaos in disorder or external field remains valid, see, Chen et al. \cite{CHHS_2015}, Chen-Sen \cite{CS17}, and Panchenko-Talagrand \cite{PT07}, the validity of chaos in temperature depends on the mixture of the $p$-spin Hamiltonians; when the model consisted of a pure $p$-spin Hamiltonian, Subag \cite{subag2017geometry} showed that chaos in temperature is absent, whereas if one couples the pure $p$-spin model with a perturbative $q$-spin Hamiltonian for $q\neq p$, Chen-Panchenko \cite{chenPan2017temperature} and Ben Arous-Subag-Zeitouni \cite{arous2020geometry} showed that temperature chaos indeed holds. Notably, the results mentioned here all assumed that the disorders are Gaussian.  The works, Auffinger-Chen \cite{AC_2016} and Chen-Lam \cite{chen2024universality}, addressed the universality of chaos for general settings of the disorders under matching Gaussian moment assumptions.

While the aforementioned works (and the current paper) measure chaos in disorder by the chaotic rearrangement of the overlap, it is also possible to formalize chaos in disorder in terms of other distances between the perturbed and unperturbed Gibbs measures. For example, El Alaoui-Montanari-Sellke \cite{AMS22_Sampling} proposed using the (normalized) 2-Wasserstein distance as a stronger notion than the overlap-based approach and showed that there is disorder chaos for the SK model in the low-temperature regime with the Wasserstein distance, which poses a barrier to stable sampling algorithms. In a subsequent work \cite{alaoui2024shattering}, the same authors demonstrated the shattering phenomenon in pure $p$-spherical spin glasses for large enough $p$ and high temperatures, which again implies disorder chaos in the Wasserstein distance. On the other hand, shattering and the overlap gap property in the Ising pure $p$-spin case in the high-temperature regime were investigated by Gamarnik-Jagannath-Kızıldağ \cite{gamarnik2023shattering}, again for large enough $p$ by making a connection to the random energy model. 

Finally, we close this subsection by mentioning that the overlap gap and disorder chaos properties have been shown to be the main features that result in computational barriers for stable algorithms to achieve near-ground states in a variety of random optimization problems, including the approximate message passing, low-degree polynomial, and Lipschitz algorithms (see \cite{GJ21overlap, GKE23algorithmic, BS22tight, BS24optimization} and references therein).



\subsection{Open Problems}
\begin{enumerate}
   
    \item For any deterministic sequence of 3-uniform hypergraphs with either a polynomial or exponential growing number of vertices, does site disorder chaos hold in the corresponding short-range model?

    \item 
    While the definition of discrete perturbation is non-ambiguous for general disorder distribution, there is more than one way to introduce continuous perturbations to disorders. For example,  
    let us consider the case when the disorders have a density  $e^{-V(x)}$ for some nice (say, continuously differentiable and convex) function $V$. Let $X$ be drawn from this distribution and let  $X(t)$ be the solution to the stochastic differential equation,
    \[ X(t) =  - V'(X(t)) dt + \sqrt{2}{dB(t)}, \ \ X(0) = X, \]
    where $(B(t))_{t\geq 0}$ is the standard Brownian motion independent of $X$.
    If $p(t, \cdot)$ is the density of $X(t)$, then $p$ satisfies the associated Fokker-Planck equation
    \[ \frac{\partial p(t, x)}{\partial t} =\frac{\partial}{\partial x} \Bigl( V'(x) p(t, x)+ \frac{\partial p(t, x)}{\partial x} \Bigr), \ \ p(0, x) = e^{-V(x)}. \]
  Note that $X(0)= X, X(t) \stackrel{d}{=} X$ for all $t \in [0, \infty)$, and $X(\infty)$ is independent of $X$. Therefore,  we can take $X(t)$ as another candidate for continuous perturbation of $X$. In the case $V(x)=x^2/2+2^{-1}\ln (2\pi)$, $(X(t))_{t \ge 0}$ is known as the Ornstein-Uhlenbeck process and it can be solved as $X(t)=e^{-t}X+e^{-t}\int_0^t e^udB(u).$ Consequently, $(X,X(t))$ corresponds to our continuous perturbation $(J,J(t)).$ It would be interesting to see if chaos in disorder can be studied beyond this case by Fourier analysis involving orthonormal polynomials in  $L^2(e^{-V(x)}dx)$. 

    \item For stable index $\alpha>1$, Theorem \ref{chaosLevy} establishes disorder chaos in the L\'evy model under the continuous perturbation. Prove disorder chaos for the discrete perturbation in the same regime of $\alpha$. Also, the case $\alpha \in (0, 1]$ remains completely open. 
    
\end{enumerate}

\section{Preliminaries}
    \subsection{Hypergraphs}\label{sec: prelim hypergraphs}
    For $\ell \in \mathbb N$, a \textit{Berge path of length $\ell$} between $v_1$ and $v_{\ell+1}$ is a sequence $(v_1,e_1,v_2,e_2,\dots,v_\ell, e_\ell, v_{\ell+1})$ of distinct $\ell+1$ vertices and distinct $\ell$ hyperedges such that $\{v_k,v_{k+1}\} \subseteq V(e_k)$ for $k=1,\dots,\ell$.
We define the \textit{distance} $d: [N] \times [N] \to \mathbb R$ between two vertices as the minimum length  among all possible Berge paths connecting those vertices, provided they are distinct; the distance is defined as $0$ if the vertices are identical. 
We assign $\infty$ if there are no paths connecting two distinct vertices.
A (closed) \textit{ball of radius $r \in \mathbb Z_+$ centered at $v\in [N]$} is defined in the usual way as $B_r(v)= \{v'\in[N]: d(v,v')\le r \}$.
Two vertices are considered \textit{connected} if there is a Berge path between them, and this establishes a partition of $[N]$ into connected components.
A \textit{Berge cycle of length $\ell$} is a sequence $(v_1,e_1,v_2,e_2,\dots,v_\ell, e_\ell, v_{\ell+1})$ where $v_1,\dots, v_\ell$ are distinct vertices, $e_1,\dots,e_\ell$ are distinct edges, and $\{v_k,v_{k+1}\} \subseteq V(e_k)$ for $k=1,\dots,\ell$ with $v_{\ell+1}=v_1$. 
A \textit{hypertree} is a connected hypergraph without Berge cycles.
For brevity, a path and a cycle will always mean a Berge path and a Berge cycle, respectively.

For a subset $B\subseteq [N]$, we define its interior and the edge boundary as
\[ \mathring B =\{e\in E : V(e)\subseteq B\}\text{ and }\partial B=\{e\in E: 0<|V(e)\cap B|<|V(e)| \}.\]  
Clearly, the edge set has the partition $E=\mathring B\cup \partial B \cup \{e\in E: |V(e)\cap B|=0\}$.

Denote $\mathbb Z_+\colonequals \mathbb N \cup \{0\}$.
For $n=(n_e)_{e\in E} \in \Z_+^{E}$, we define \begin{equation*}
    E(n)=\{e\in E: n_e>0\} \text{ and } V(n)= \{v \in [N]: v\in V(e) \text{ for some } e\in E(n)\}.
\end{equation*} Let $G(n)= (V(n),E(n))$ be the associated sub-hypergraph of $G$. 
Denote the connected component of $v\in V(n)$ in $G(n)$ as $C(v,n)$ for any $n\in \mathbb Z_+^E$.

\subsection{Hermite Polynomial Expansion}\label{sec: hermite polynomial expansion}
Let $(h_m)_{m\ge 0}$ be the orthonormal basis of normalized Hermite polynomials for $L^2(\mu)$ where $\mu$ is the standard Gaussian measure on $\mathbb R$ and $h_0 \equiv 1$. Then, products of the form $h_n(x)\colonequals \prod_{e\in E}h_{n_e} (x_e) $ for $n\in \mathbb Z_+^E$ and $x\in\mathbb R^E$ is a basis for $L^2(\mu^{\otimes E})$. For any $\phi \in L^2(\mu^{\otimes E})$,  we have the Hermite polynomial expansion \begin{equation}\label{eq: chaos expansion}
\phi(x)=\sum_{n\in \mathbb Z_+^E}\hat \phi(n) h_n(x), \quad x\in \mathbb R^E,\end{equation} where \[\hat \phi(n)\colonequals \e (\phi(J)h_n(J)).\]    
We use the notation $|n|\colonequals \sum_{e\in E}n_e$.

\begin{lemma}\label{lem: semigroup: main}
 Let $t\ge 0$. For any $\phi \in L^2(\mu^{\otimes E})$, we have \begin{align}
    \e \bigl( \phi(J) \phi(J(t)) \bigr) &=  \sum_{n\in \mathbb Z_+^E}e^{-|n|t} \hat \phi (n)^2, \label{eq: semigroup: continuous}\\
    \e \bigl( \phi(J) \phi(J(t)) \bigr) &=  \sum_{n\in \mathbb Z_+^E} e^{-|E(n)|t}  \hat \phi (n)^2  \label{eq: semigroup: discrete}
 \end{align}
     for the continuous and discrete perturbations, respectively.

\end{lemma}
\begin{proof}
Our proof relies on the following two equations:
for the continuous perturbation, 
for any $n\in \mathbb Z_+^E,$ we have \begin{equation}\label{eq: Semigroup action}
    \e (h_n(J(t))|J)= \sum_{n\in \mathbb Z_+^E}e^{-|n|t} h_n(J),
\end{equation}
whereas for the discrete perturbation, 
\begin{equation}\label{eq: Semigroup action: discrete}
     \e (h_n(J(t))|J)= \sum_{n\in \mathbb Z_+^{E}}e^{- |E(n)| t} h_{n}(J).
\end{equation}
    Equation \eqref{eq: Semigroup action} is a standard result of the OU process (see, e.g., \cite{Chatterjeebook}).
    Equation \eqref{eq: Semigroup action: discrete} can be obtained by taking expectation in the Bernoulli randomness.

    Now, let $t\ge 0$ and consider the continuous perturbation. 
    From \eqref{eq: chaos expansion} and \eqref{eq: Semigroup action}, it holds that \begin{equation*}
        \e ( \phi(J(t))|J) = \sum_{n \in \mathbb Z_+^E}\hat \phi(n) e^{-|n|t}h_n(J).
    \end{equation*}
    Since $\e (h_n(J) h_m(J)) =\1_{n=m}$ for $n,m\in \mathbb Z_+^E$, we have
  \[
        \e ( \phi(J) \phi(J(t)) ) = \e ( \phi(J) \e( \phi(J(t))|J) )= \sum_{n\in \mathbb Z_+^E} \hat \phi (n)^2e^{-|n|t}.   
    \]
    The case of discrete perturbation can similarly be treated using \eqref{eq: Semigroup action: discrete}.
    
\end{proof}


\subsection{Geometric Conditions for Vanishing Fourier-Hermite Coefficients}\label{sec: deterministic lemma}

 This subsection consists of the main ingredient in the present work, in which we establish the promised elementary algebraic equation for the Fourier-Hermite series coefficients, see Lemma \ref{lem: coefficient zero} below, and deduce
a sufficient condition for determining the zero coefficients in a systematic way, see Propositions \ref{prop: lowfrequencydies} and \ref{prop: hypertree lemma} below. In particular,
Lemma \ref{lem: coefficient zero} and Proposition \ref{prop: lowfrequencydies} are in lieu of Lemma 1 and its geometric consequence Corollary 5 in \cite{Chatterjee}, respectively.
While \cite{Chatterjee}'s trick involves flipping a subset of disorders and spins, as we will see below, flipping only spins suffices for our purposes.
Nevertheless, we also present Proposition \ref{prop: conditional expectation} as a generalization of the trick in Lemma~1 of \cite{Chatterjee}.

Recall the notation $\eqref{eq: product notation}.$
For any $a, b\in \{\pm 1\}^N$, we define the tensors ${\bf a} =(a_e)_{e\in E}$, ${\bf b}=(b_e)_{e\in E}$, and define their Hadamard product ${\bf a}\circ {\bf b}= (a_e b_e)_{e\in E}$. 
We denote the Gibbs average of the disorder $(a_e J_e)_{e\in E}\equalscolon {\bf a}\circ J$ by $\la \cdot \ra_{{\bf a}\circ J}$.


\begin{lemma}\label{lem: coefficient zero}
Consider the generalized short-range model with Hamiltonian in \eqref{hamiltonian} with $\beta \in [0, \infty)$.
For any $a\in \{\pm 1\}^N$, $i,j\in[N]$, and $n\in \mathbb Z_+^E$, it holds that
\begin{equation*}
	\e \bigl( h_n(J)\la \sigma_i \sigma_j \ra \bigr)=I_n(a) \e \bigl( h_n(J)\la \sigma_i \sigma_j\ra \bigr),
\end{equation*}
where 
\[
I_n(a):=a_i a_j\prod_{e\in E(n)}a_e^{n_e} .
\]
\end{lemma}

\begin{remark}
 It is immediate from the lemma that $\e (h_n(J)\la \sigma_i \sigma_j \ra) =0$ if $I_n(a)=-1$ for some choice of $a\in \{\pm 1\}^N$. However, the following example demonstrates that the converse is not true, even for a connected graph ($p=2$).
    
    Let $G = (V, E) = ([4],\{\{1,2\}, \{1,3\}, \{2,4\}\}) $. Let $\phi = \la \sigma_1 \sigma_2 \ra$. Then one can check that $\phi$ does not depend on $J_{\{1,3\}}, J_{\{2,4\}}$. Indeed,
\begin{align*}
    \phi &= \frac{\sum_{\sigma} \sigma_1 \sigma_2 (1+ \tanh(\beta J_{\{1,2\}}) \sigma_1 \sigma_2 ) (1+ \tanh(\beta J_{\{1,3\}}) \sigma_1 \sigma_3 ) (1+ \tanh(\beta J_{\{2,4\}}) \sigma_2 \sigma_4 ) }{
    \sum_{\sigma} (1+ \tanh(\beta J_{\{1,2\}}) \sigma_1 \sigma_2 ) (1+ \tanh(\beta J_{\{1,3\}}) \sigma_1 \sigma_3 ) (1+ \tanh(\beta J_{\{2,4\}}) \sigma_2 \sigma_4 ) } \\
    &= \frac{\sum_{\sigma_1, \sigma_2} \sigma_1 \sigma_2 (1+ \tanh(\beta J_{\{1,2\}}) \sigma_1 \sigma_2 )  }{\sum_{\sigma_1, \sigma_2} (1+ \tanh(\beta J_{\{1,2\}}) \sigma_1 \sigma_2 )  } = \tanh(\beta J_{\{1,2\}}).
\end{align*}
Now let $n \in \mathbb{Z}_+^E$ with $n_{\{1,2\}}=1,$  $n_{\{1,3\}}=2,$ and $n_{\{2,4\}}=0$.
    It follows that 
    \[ \hat \phi(n) = \e [ \phi h_{1}(J_{\{1,2\}})h_{2}(J_{\{1,3\}})] = \e [ \phi h_{1}(J_{\{1,2\}})] \e[ h_{2}(J_{\{1,3\}})] = 0, \]
    since $\e[ h_{2}(J_{\{1,3\}})] = 0$. On the other hand, 
    \[ I_n(a)= a_1a_2 (a_1 a_2)(a_1 a_3)^2 =1 \ \text{ for all }  a\in \{\pm 1\}^4. \]

\end{remark}

\begin{proof}[\bf Proof of Lemma~\ref{lem: coefficient zero}]
	Note that for any deterministic $a\in \{-1,1\}^N,$ since $\rho_p$'s are odd functions, we can write
	\[
		H_J( {\bf a} \circ\sigma)=\sum_{2\le p \le \Delta}\sum_{e\in E_p}\rho_p(J_{e})a_{e}\sigma_{e} =\sum_{2\le p \le \Delta}\sum_{e\in E_p}\rho_p(a_{e} J_{e})\sigma_{e} = H_{{\bf a}\circ J}(\sigma)
        \]
so that $\la \sigma_i \sigma_j\ra_{ J}=a_i a_j\la \sigma_i \sigma_j\ra_{{\bf a}\circ J}.$ 
	From this and the fact that ${\bf a}\circ {\bf a}\circ J=J$,
	\begin{equation*}
	    \e [h_n(J)\la \sigma_i \sigma_j \ra_J]=a_ia_j\e [h_n({\bf a} \circ {\bf a}\circ J) \la \sigma_i\sigma_j\ra_{{\bf a}\circ J})]= a_ia_j\e [h_n({\bf a}\circ J)\la \sigma_i\sigma_j\ra_{J}],
	\end{equation*}
where the second equality used the obvious fact that ${\bf a}\circ J\stackrel{d}{=}J$.
Finally, using $h_n({\bf a}\circ J)=\prod_{e\in E}h_{n_{e}}(a_e J_{e})= \prod_{e\in E} a_e^{n_{e}}h_{n_{e}}(J_{e})=\prod_{e\in E}a_e^{n_e}\cdot h_n(J)=\prod_{e\in E(n)}a_e^{n_e}\cdot h_n(J)$, we have
\[	\e h_n(J)\la \sigma_i \sigma_j \ra_J=a_i a_j \prod_{e\in E(n)}a_e^{n_{e}}\cdot 	\e h_n(J)\la \sigma_i \sigma_j \ra_J.
\]
\end{proof}

Now, for any $i,j\in [N]$, define \[\phi_{ij} (J) = \la \sigma_i \sigma_j \ra,\] which is a bounded measurable function of the disorder $J$, hence in $L^2(\mu^{\otimes E})$.

\begin{proposition}\label{prop: lowfrequencydies}
Consider the mixed even $p$-spin short-range model on $G=([N],E)$. Let $i,j$ be distinct vertices in $[N]$ and $n\in \mathbb Z_+^E$. If $\hat \phi_{ij}(n)\ne 0$, then we have $i,j \in V(n)$ and there exists a path from $i$ to $j$ in $G(n)$. In particular, $|E(n)|\ge d(i,j)$.
\end{proposition}

\begin{proof}
Recall the notations in Lemma \ref{lem: coefficient zero}.
We will choose appropriate $a$'s satisfying $I_n(a)=-1$ and invoke Lemma \ref{lem: coefficient zero}.

Suppose $i\notin V(n)$. In this case, let $a_i=-1$ and $a_l=1$ if $l\neq i$. 
In particular, $a_j=1$. 
Moreover, $i\notin V(n)$ implies $a_e=1$ for any $e\in E(n)$, so $I_n(a)=a_i a_j\prod_{e\in E(n)}a_e^{n_e}= -1 $.
In this case, we conclude $\hat \phi_{ij}(n)=0$ by Lemma \ref{lem: coefficient zero}.

Now, suppose $i\in V(n).$
Further assume that $j\notin C(i,n),$ i.e., there is no path from $i$ to $j$ in the sub-hypergraph $G(n)$.
Let $a_l=-1$ if $l\in C(i,n)$ and $a_l=1$ otherwise.
In particular, 
\begin{equation}\label{eq: assign a's}
    a_i=-1, \quad a_j=1.
\end{equation}
We claim that $a_e=1$ whenever $e\in E(n)$.
To that end, fix an edge $e\in E(n)$.
There are two cases to consider.
If $e\in \mathring C(i,n)$, since every edge in $E$ has even number of vertices, we have $a_e=(-1)^{p}=1$  for some even $p\in 2\Z_+$.
If $e\notin \mathring C(i,n)$, we must have $V(e)\cap C(i,n)=\emptyset$ since otherwise, $e\in E(n)$ implies $e\in \mathring C(i,n)$, a contradiction.
Thus,  $a_e=1$ trivially holds, and this finishes the proof of the above claim.
Combining \eqref{eq: assign a's} and the previous claim,  we have $I_n(a)=a_i a_j \prod_{e\in E(n)} a_e^{n_e}=-1$, and $\hat \phi_{ij}(n)=0$ by Lemma \ref{lem: coefficient zero}.

Changing the roles of $i$ and $j$, we conclude that $\hat \phi_{ij}(n)\ne 0$ implies $i,j\in V(n)$ and they are in the same component of $G(n)$. 
\end{proof}


If we impose that the hypergraph is locally a hypertree, we may allow hyperedges with odd number of vertices. The following Proposition will be used in Section \ref{sec: diluted mixed p} for the proof of the disorder chaos of the diluted mixed $p$-spin model, 


\begin{proposition}\label{prop: hypertree lemma}
   Consider the generalized short-range model on $G= ([N], E)$. Let $i, j$ be distinct vertices in $[N]$ and $n\in \mathbb Z_+^E$.
   Assume that the ball of depth $r \ge 1$ around $i$ in $G$ is a hypertree.  If $\hat \phi_{ij}(n)\neq 0$, then $i,j \in V(n)$ and $|E(n)| \ge \min(r, d(i, j)).$ 
\end{proposition}

\begin{proof}
We note that $I_n(a)$ for $\la \sigma_i \sigma_j \ra$ does not depend on edges with even $n_e$. Keeping that in mind, for $n\in \mathbb Z_+^E$, we let $E_{\odd}(n)\colonequals \{e\in E(n): \textrm{$n_e$ is odd}\}$, 
$V_{\odd}(n)\colonequals \bigcup_{e\in E_{\odd}(n)}V(e) $ be the associated vertices, and $G_{\odd}(n)=(V_{\odd}(n), E_{\odd}(n))$ be the associated sub-hypergraph.
Note that $V_{\odd}(n)\subseteq V(n)$.

If $i\notin V_{\odd}(n)$, then we have $\hat \phi_{ij}(n)=0$ by choosing $a_i=-1$ and $a_l=1$ for $l\ne i$. The same conclusion holds if $j\notin V_{\odd}(n)$.

Suppose $i,j\in V_{\odd}(n)$ and let $|E(n)|< \min(r, d(i,j))$. 
Then, the connected component $C_{\odd}(i,n)$ of $i$ in the sub-hypergraph $G_{\odd}(n)$ is a hypertree and $j\notin C_{\odd}(i,n)$.
Choose a \textit{hyperleaf} $v\in C_{\odd}(i,n)$, i.e., there exists exactly one edge in $E_{\odd}(n)$ containing $v$.  
Since all hyperedges contain at least two vertices, we may assume that $ v\neq i$.
Define $a_v=-1$ and $a_l=1$ if $l \ne v$. 
Then, $a_i a_j \prod_{e\in E(n)}a_e^{n_e}=-1$, and Lemma \ref{lem: coefficient zero} yields $\hat \phi_{ij}(n)=0$.

\end{proof}

\begin{remark}
Proposition~\ref{prop: lowfrequencydies} does not always hold if we allow hyperedges with odd number of vertices. For a counterexample, consider the hypergraph $G = (V, E)$ with vertex set $V = \{0, 1', 2', 3', 1'', 2'', 3''\}$ and hyperedge set $E = E' \cup E'' \cup \{\{0, 3', 3''\} \}$ for
\begin{align*}
    E' = \{ \{ 1', 2', 3' \}, \{ 2', 3' \} \} \,\,\mbox{and}\,\,E'' = \{ \{ 1'', 2'', 3'' \}, \{ 2'', 3'' \} \}.
\end{align*}
See Figure~\ref{fig:counterexample}. Let $\phi = \la \sigma_{1'} \sigma_{1''} \ra.$
\begin{figure}[h] 
    \centering
        \includegraphics[width=0.40\textwidth]{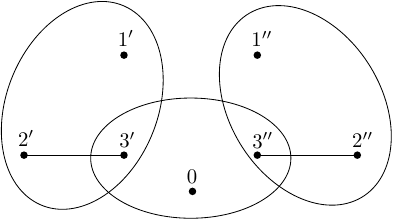} 
 \caption{Hypergraph $G$.} \label{fig:counterexample}
\end{figure}    
Below we write $\sigma = (\sigma_0, \sigma', \sigma'') \in \{-1, 1\}^V $ with 
$\sigma' \in  \{-1, 1\}^{\{1', 2', 3' \}}$ and $\sigma'' \in  \{-1, 1\}^{\{1'', 2'', 3'' \}}$.
Note that $\phi$ is the ratio of the following two quantities:
\begin{align*}
 \sum_{\sigma} \sigma_{1'} \sigma_{1''}  \prod_{ e \in E }   ( 1+ \tanh(\beta J_{e}) \sigma_e) &= \sum_{\sigma', \sigma''} \sigma_{1'} \sigma_{1''}  \prod_{ e \in E' \cup E'' }   ( 1+ \tanh(\beta J_{e}) \sigma_e), \\
\sum_{\sigma} \prod_{ e \in E }   ( 1+ \tanh(\beta J_{e}) \sigma_e) &= \sum_{\sigma', \sigma''}\prod_{ e \in E' \cup E'' }   ( 1+ \tanh(\beta J_{e}) \sigma_e),
  \end{align*}
where the two equalities hold by summing over the single spin $\sigma_0.$ Therefore,  $\phi$ decouples as $\phi = \la \sigma_{1'} \ra_{E'}\la \sigma_{1''} \ra_{E''}$ with
\begin{align}\label{add:eq0}
    \la \sigma_{1'} \ra_{E'} &= \frac{\sum_{\sigma'} \sigma_{1'} ( 1+ \tanh(\beta J_{\{1', 2', 3'\}}) \sigma_{1'} \sigma_{2'} \sigma_{3'}) ( 1+ \tanh(\beta J_{\{2', 3'\}}) \sigma_{2'} \sigma_{3'} )  }
    {\sum_{\sigma'}  ( 1+ \tanh(\beta J_{\{1', 2', 3'\}}) \sigma_{1'} \sigma_{2'} \sigma_{3'}) ( 1+ \tanh(\beta J_{\{2', 3'\}}) \sigma_{2'} \sigma_{3'} ) }
\end{align}
and $ \la \sigma_{1''} \ra_{E''}$ is defined similarly. Distributing the product over the sum, we deduce that 
\begin{align*}
    \la \sigma_{1'} \ra_{E'} = \tanh(\beta J_{\{1', 2', 3'\}})\tanh(\beta J_{\{2', 3'\}}), \quad
    \la \sigma_{1''} \ra_{E''} = \tanh(\beta J_{\{1'', 2'', 3''\}})\tanh(\beta J_{\{2'', 3''\}}).
\end{align*}
 If we take $n \in \mathbb{Z}_+^E$ with $n_{\{1', 2', 3'\}} =n_{\{2', 3'\}} = n_{\{1'', 2'', 3''\}} =n_{\{2'', 3''\}} = 1 $ and $n_{\{0, 3', 3''\}} = 0$, then 
 \[ \hat \phi(n) = \e [\la \sigma_{1'} \ra_{E'} \la \sigma_{1''} \ra_{E''} J_{\{1', 2', 3'\}} J_{\{2', 3'\}}
 J_{\{1'', 2'', 3''\}} J_{\{2'', 3''\}} ]  = (\e[ J \tanh(\beta J) ])^4 >0,   \]
 for any $\beta >0$. However, there is no path between $1'$ and $1''$ in $E(n)$, since any such path cannot avoid the hyperedge $\{0, 3', 3''\}$. Let us also point out that this example does not contradict Proposition~\ref{prop: hypertree lemma} since $1$-neighborhood of vertex $1'$ contains a Berge cycle, namely, $(2', \{ 1', 2', 3'\}, 3', \{  2', 3'\}, 2').$

 One can  replace the hyperedge $\{0, 3', 3''\}$ with a path of arbitrary length of $k$ connecting $3$ and $3'$ such as $\{3', x_1, y_1\}, \{y_1, x_2, y_2\},  \{y_2, x_3, y_3\}, \ldots, \{y_{k-1}, x_k, 3''\}$ where $1', 2', 3', 1'', 2'', 3''$, $x_1, \ldots, x_k$, $y_1, \ldots, y_{k-1}$ are all distinct vertices. Summing over 
 spins $\sigma_{x_1}, \ldots, \sigma_{x_k}$ as we did in \eqref{add:eq0}, we can still conclude that $\phi = \la \sigma_{1'} \ra_{E'}\la \sigma_{1''} \ra_{E''}$ and $\hat \phi(n) >0$ for the index $n$  where   $n_e = 0 $ on this new path and $n_e =1$ on the rest of the four hyperedges. This shows that if we allow odd hyperedges, then it is possible that $\hat \phi(n)$ is non-zero, but $d(i, j)$ is arbitrarily large (but finite) compared to $|E(n)|$.
 Also, we can have $d(i,j)=\infty$ and $\hat \phi(n) >0$ by simply deleting the hyperedge $\{0,3',3''\}$ from $E$.

 \end{remark}

We present a proposition that is of independent interest, but is not used elsewhere in the paper.
\begin{proposition}\label{prop: conditional expectation}
Consider the EA model.
Let $G= ([N], E)$ be a graph and $i \ne j \in [N]$. If $S \subseteq E$ does not contain a path in $S$ that connects $i$ and $j$, then 
$$
\e[\la \sigma_i \sigma_j \ra| J_r,r\in S]=0.
$$
\end{proposition}
\begin{proof}
    Let $\phi =\la \sigma_i \sigma_j \ra $. Then 
 \[  \e[ \phi | J_r,r\in S] = \sum_{n\in \mathbb{Z}_+^E, n|_{E\setminus S}=0 } \hat \phi(n) h_n(J).  \]
 Now for any $n|_{E\setminus S}=0$, the edge set $E(n) = \{ e \in E: n_e >0\} \subseteq S$ and hence, it does not contain a path joining $i$ and $j$. Therefore, $\hat \phi(n) = 0$. The proof follows.
\end{proof}

\section{Proof of Theorem \ref{thm: mixed p spin disorder chaos}}\label{sec: proof of generalized short-range model}
    

\begin{proof}[{\bf Proof of Theorem \ref{thm: mixed p spin disorder chaos}}]
It suffices to prove the theorem for finite $\beta \ge 0$.  The case $\beta 
 = \infty$ would then immediately follows from sending $\beta \to \infty$. 
From Lemma \ref{lem: semigroup: main}, Proposition \ref{prop: lowfrequencydies}, and Parseval's identity $\sum_{n\in \mathbb Z_+^E}\hat \phi_{ij}(n)^2 = \e \la \sigma_i\sigma_j\ra^2 $, and the fact that $|n|\ge |E(n)|$ for $n\in\mathbb Z_+^E$, we have for both types of perturbations 
\[
    \e \la \sigma_i\sigma_j\tau_i \tau_j\ra_t \le e^{-t d(i,j)}\e \la \sigma_i \sigma_j \ra^2 \]  for any $i,j\in [N]$. Therefore, for any $r\in \mathbb Z_+$,
    \begin{align}
    \e \la R(\sigma,\tau)^2\ra_t  &= \frac{1}{N^2} \sum_{1\le i,j\le N} \e \la \sigma_i\sigma_j\tau_i \tau_j\ra_t \le \frac{1}{N^2}\sum_{1\le i,j\le N } e^{-t d(i,j)} \e \la \sigma_i \sigma_j\ra ^2 \nonumber\\
    &\le \frac{1}{N}\max_{i\in [N]}\sum_{j: d(i,j) \le r}e^{-td(i,j)}\e \la \sigma_i\sigma_j\ra^2  + \frac{1}{N^2}\sum_{(i,j):d(i,j) >  r}e^{-td(i,j)}\e \la \sigma_i\sigma_j\ra^2 \label{eq:intermediate_bd}\\
    &\le \frac{1}{N}\max_{i\in [N]}|B_{ r }(i)| + e^{-t  r  }\e \la R(\sigma^1,\sigma^2)^2\ra \label{eq:intermediate_bd:2},   
    \end{align}
    which proves \eqref{eq:general_bd} since $|R(\sigma^1,\sigma^2)| \le 1$.
    
Under (i) of Assumption \ref{edge set assumption: mixed p spin}, setting $r=N+1$ so that the second term vanishes in \eqref{eq:intermediate_bd}, we have
\begin{align}
    \e \la R(\sigma,\tau)^2\ra_t &\le \frac{1}{N}\max_{i\in[N]}\sum_{1\le j \le N}e^{-td(i,j)}
    \le \frac{1}{N} +  \frac{1}{N}\max_{i\in[N]}\sum_{ j \in [N] \setminus \{i\}}e^{-td(i,j)},  \label{eq: polynomial}
\end{align}
which we bound below by grouping vertices according to the values of $\lfloor td(i,j) \rfloor,$  \begin{align*}
\frac{1}{N} +  \frac{1}{N}\max_{i\in [N]}\sum_{k=0}^\infty \sum_{\substack{j: \lfloor td(i,j) \rfloor=k\\
j \ne i}} e^{-k}   &\le \frac{1}{N} +  \frac{1}{N} \sum_{k=0}^\infty e^{-k}\max_{i\in[N]}|\{j \ne i: kt^{-1} \le d(i,j)<(k+1)t^{-1}\}|
\end{align*}
Note that $|\{j \ne i : kt^{-1} \le d(i,j)<(k+1)t^{-1}\}| \le A(k+1)^{\theta}t^{-\theta},$ which also holds true if $(k+1) < t$, since in this  case the set becomes empty. 

Thus, \eqref{eq: polynomial} is bounded by
\begin{equation*}
      \frac{1}{N}  + \sum_{k = 0 }^\infty  e^{-k} A(k+1)^{\theta} t^{-\theta}
      \le  \frac{1}{N} + \frac{C}{Nt^{\theta}},
\end{equation*}
for some constant $C>0$ depending only on $A$ and $\theta$.

Under (ii) of Assumption \ref{edge set assumption: mixed p spin}, then \eqref{eq:intermediate_bd:2} is bounded above by \begin{align*}
    \e \la R(\sigma,\tau)^2\ra_t &\le \frac{1}{N}\max_{i\in[N]}|B_{ r}(i)| +e^{-tr} 
    \le \frac{1}{N} A\gamma^r+e^{-tr}.\nonumber
\end{align*}
  Setting $r= \lceil (t+ \log \gamma)^{-1}\log  N \rceil$, we obtain the bound 
  \[ \e \la R(\sigma,\tau)^2\ra_t \le \frac{C}{N^{t/(t+\log \gamma)}}  \] for some $C>0$ depending only on $A$ and $\gamma$.

\end{proof} 
 \begin{remark} 
In the exponential growth case, we can retain $\e \la R(\sigma^1,\sigma^2)^2\ra$  in the bound \eqref{eq:intermediate_bd:2} to obtain
\[ \e \la R(\sigma,\tau)^2\ra_t \le \frac{1}{N} A\gamma^r+e^{-tr} \e \la R(\sigma^1,\sigma^2)^2\ra. \]
By taking $r= \lceil (t+ \log \gamma)^{-1}\log ( N  \e \la R(\sigma^1,\sigma^2)^2\ra) \rceil$, we arrive at the bound 
\[ \e \la R(\sigma,\tau)^2\ra_t \le C  (N \e \la R(\sigma^1,\sigma^2)^2\ra) ^{-t/(t+\log \gamma)} \e \la R(\sigma^1,\sigma^2)^2 \ra \]
for some constant $C$ depending on $A$ and $\gamma$. Of course, this bound offers no improvement 
over \eqref{eq: result for condition (ii)} if $ \e \la R(\sigma^1,\sigma^2)^2 \ra \ge c$ for some positive constant $c$. However, in the case when $\e \la R(\sigma^1,\sigma^2)^2\ra \ll 1$, instead of the trivial bound $ \e \la R(\sigma,\tau)^2\ra_t \le \e \la R(\sigma^1,\sigma^2)^2\ra$, we can conclude from  the above bound that
\[ \e \la R(\sigma,\tau)^2 \ra_t  \ll \e \la R(\sigma^1,\sigma^2)^2\ra,\]
if $t \gg 1 / \log (N \e \la R(\sigma^1,\sigma^2)^2\ra).$ 
 \end{remark}

\section{Proof of Theorem \ref{thm: diluted mixed p spin}}\label{sec: proof of diluted mixed p spin model}

We prove Theorem \ref{thm: diluted mixed p spin} using a spectral method similar to the proof of Theorem \ref{thm: mixed p spin disorder chaos}. 
The main difference lies in the utilization of Proposition \ref{prop: hypertree lemma} instead of Proposition \ref{prop: lowfrequencydies}.
Due to the diluted nature of our hypergraph, we demonstrate in Lemma \ref{lem: nbd hypertree} that the hypergraph looks locally like a hypertree up to depth $(2\log \lambda)^{-1}(1-\eps) \log N$ with probability at least $1-N^{-\eps}$ for any $0<\eps<1$.
We achieved this via the use of the exploration process defined in Section \ref{sec: exploration process}, which allows us to control the growth rate and the probability of producing a cycle at each depth of a ball. The use of the exploration process to study the component size and their properties of random hypergraphs is standard, see for example, \cite{cooley2018size, PZ19, schmidt1985component}.  However, we could not locate Lemma~\ref{lem: nbd hypertree} in the literature in its exact form,  and we chose to include its proof for completeness.  Combined with Proposition~\ref{prop: hypertree lemma}, we show that the low-frequency coefficients in the Hermite polynomial expansion are zero with high probability.
The remainder of the proof follows as in Theorem \ref{thm: mixed p spin disorder chaos}.

\subsection{Exploration Process}\label{sec: exploration process}
We first describe the exploration process on the random hypergraph which starts from a single vertex and conducts a breadth-first search on its connected component. Fix a vertex $i\in [N]$.
We define the \textit{exploration process} $(R_t ,I_t ,S_t ,E_t )_{t\in \mathbb Z_+} \in [N]^3 \times E$ inductively as follows. 
Put $R_0 =\emptyset$, $I_0 =\{i\}$, $S_0 =[N]\setminus\{i\}$, and $E_0 =\emptyset$.  
Given $\mathcal F_{t}\colonequals \sigma(R_s ,  I_s , S_s ,E_s )_{0\le s \le t}$, we define 
\begin{align*}
    R_{t+1} &= R_t \cup I_t , \nonumber\\
    E_{t+1} &=\{e\in E: V(e)\cap I_t  \ne \emptyset, \ V(e)\cap R_t =\emptyset\}, \nonumber\\
    I_{t+1} &= \bigcup_{e\in E_{t+1} } V(e)\cap S_t ,\nonumber\\
    S_{t+1} &= S_t \setminus I_{t+1} . 
\end{align*} Heuristically, at step $t$, $R_t $ represents the removed (or recovered) vertices, $I_t $ represents the newly infected (or explored) vertices, $S_t $ represents the susceptible (or unexplored) vertices. $E_t$ represents the set of hyperedges that we `reveal' during the step $t$  of the algorithm. It consists of hyperedges with the property that they contain at least one vertex from $I_{t-1}$ and all of their vertices belong to $I_{t-1} \cup S_{t-1}$.


Below we state some simple properties of the exploration process without a proof.
Here $\dot \cup$ denotes the disjoint union.
\begin{proposition}[Properties of exploration processes]\label{prop: properties of the exploration process}
\indent
    \begin{enumerate}[label=(\roman*)]
        \item $R_t =\dot \bigcup_{s<t}I_s $ for any $t\in \mathbb Z_+.$
        \item  $R_t \dot\cup I_t \dot\cup S_t  = [N]  $ for any $t\in \mathbb Z_+$.
        \item $V(e) \subseteq I_{t} \cup I_{t+1} $ for any $t\in \mathbb Z_+$ and $e\in E_{t+1} $.
        \item $\bigcup_{t\in\mathbb Z_+}R_t = V(C)$  and  $\dot \bigcup_{t\in \mathbb Z_+}E_t  = E(C)$ where $C$ is the connected component of $i$.
        \item $B_{t}  = R_{t+1}  = \dot \bigcup_{s \le t}I_s  $ for any $t\in \mathbb Z_+$, where $B_{t} = B_t(i)$ is the ball of radius $t$ centered at $i$.
    \end{enumerate}
\end{proposition}

Recall $\lambda =  \sum_{2\le p \le \Delta} p(p-1) \alpha_p$. Let us also define  $\lambda' =  \sum_{2\le p \le \Delta} p(p-1)(p-2) \alpha_p$.
\begin{lemma}\label{lem: expectation of infected}
For $t\in\mathbb Z_+$, we have \begin{align*}
        \e |I_t | &\le \lambda^t,\\
        \e |I_t |^2 &\le  \sum_{k=t}^{2t}\lambda^{k}+\lambda'  \sum_{k=t-1}^{2(t-1)}\lambda^k.
    \end{align*}
\end{lemma}
\begin{proof}
    Let $t\in\mathbb Z_+$.
      Observe that $v\in I_{t+1} $ only if there is a $p$-hyperedge in $E_{t+1} $ connecting $v$ to some vertex in $I_t $ for some $2\le p \le \Delta$, and similarly,  for $v\ne v' \in S_t $, $\{v, v'\} \subseteq I_{t+1} $ only if either there is a $p$-hyperedge in $E_{t+1} $ connecting $\{v,v'\}$ to some vertex in $I_t $ for some $2\le p \le \Delta$ or there is a pair $(e_1,e_2)$ of a $p_1$-hyperedge $e_1\in E_{t+1}  $ and a $p_2$-hyperedge $e_2\in E_{t+1}  $ for some $2\le p_1,p_2\le \Delta$ such that $v\in V(e_1)\setminus V(e_2)$ and $v'\in V(e_2)\setminus V(e_1)$. 
 \begin{figure}[h]
    \centering\label{figure 1}
    \begin{minipage}{0.43\textwidth}
        \centering
        \includegraphics[width=0.65\textwidth]{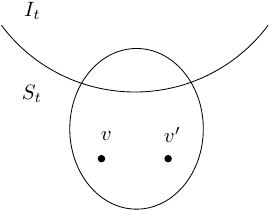} 
    \end{minipage}\hfill
    \begin{minipage}{0.55\textwidth}
        \centering
        \includegraphics[width=0.65\textwidth]{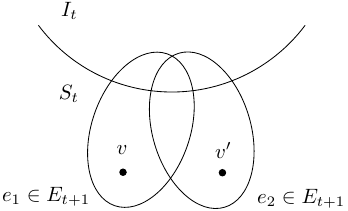} 
    \end{minipage}
    \caption{(Left panel): $v, v'$ belonging to the same hyperedge. (Right panel): $v, v'$ belonging to different hyperedges.}
\end{figure}

Hence, \begin{align}\label{eq: stochastic dominance: 0}
        \p(v\in I_{t+1} |\mathcal F_t)&\le q_{t}, \nonumber
        \\ \p(v, v'\in I_{t+1} |\mathcal F_t)&\le  q'_t,  
        \end{align}
         where \begin{align*}
             q_t&\colonequals  1- \prod_{2\le p \le \Delta}\Bigl(1-\frac{\alpha_p N}{{N\choose p}}\Bigr)^{{N-2 \choose p-2} |I_t |}  \\
             q'_t &\colonequals 1- \prod_{3\le p \le \Delta}\Bigl(1-\frac{\alpha_p N}{{N\choose p}}\Bigr)^{{N-3 \choose p-3} |I_t |}+ \Bigl(1- \prod_{2\le p \le \Delta}\Bigl(1-\frac{\alpha_p N}{{N\choose p}}\Bigr)^{{N-3 \choose p-2} |I_t |}\Bigr)^2. 
         \end{align*}
         Using the elementary bound $1-\prod_{2\le p \le \Delta}(1-x_p)^{n_p}\le \sum_{2\le p \le \Delta}n_p x_p$ for $n_p \in \mathbb Z_+$ and $x_p\in (0,1),$ which follows from the convexity of the function $t\mapsto \prod_{2\le p \le \Delta}(1-tx_p)^{n_p}$ on the interval $(0,1)$, 
        \begin{align*}
            q_t &\le \sum_{2\le p \le \Delta}\frac{\alpha_p N}{{N \choose p}} {N-2 \choose p-2} |I_t|= \frac{\lambda|I_t|}{N-1},\\
            q'_t &\le \sum_{3\le p \le \Delta}\frac{\alpha_p N}{{N\choose p}}{N-3\choose p-3}|I_t|+ \Bigl(\sum_{2\le p \le \Delta}\frac{\alpha_p N}{{N\choose p}}{N-3\choose p-2}|I_t|\Bigr)^2 \le \frac{\lambda'|I_t|}{(N-1)(N-2)}+  \frac{\lambda^2|I_t|^2}{(N-1)^2}.
        \end{align*}
         From the previous display and the trivial bound $|S_t |\le N-1$, we have \begin{align*}
        q_{t}  |S_{t} | &\le \lambda |I_{t} |,
        \\q_t' |S_t | (|S_t|-1) &\le \lambda' |I_t | + \lambda^2 |I_t |^2.
    \end{align*} 
       Now, since $|I_{t+1} |= \sum_{v\in S_t }\1_{\{v\in I_{t+1} \}}$, \eqref{eq: stochastic dominance: 0} implies the recursions
\begin{align}
    \e |I_{t+1} | &\le q_{t} \e  |S_t | \le \lambda \e |I_t|, \label{eq: recursion: 1}\\
    \e |I_{t+1} |^2 &\le  q_{t} \e  |S_t | +q'_{t} \e  |S_t| (|S_t|-1)\le (\lambda+\lambda') \e |I_t | + \lambda^2\e |I_t |^2 .\label{eq: recursion: 2}
\end{align}
Equation \eqref{eq: recursion: 1} and the initial condition $|I_0|=1$ readily implies $\e|I_t|\le \lambda^t$.
Plugging this into \eqref{eq: recursion: 2}, we have \[ \e |I_{t+1} |^2 \le  (\lambda+\lambda') \lambda^t + \lambda^2\e |I_t |^2 ,\] and an induction easily yields \[\e|I_{t}|^2\le (\lambda+\lambda')\sum_{k=0}^{t-1}\lambda^{t-1+k} +\lambda^{2t}.\]
\end{proof}

\begin{lemma}\label{lem:vol_growth_random_hypg}
Assume that $\lambda > 1$. Then there exists a constant $C$ depending only on $(\alpha_p)_{2\le p \le \Delta}$ such that
for any $t \in \mathbb{Z}_+$, 
\[ \e | B_t(i)| \le C \lambda^{t}. \]
\end{lemma}
\begin{proof}
    By Proposition~\ref{prop: properties of the exploration process}(v), $| B_t(i)|  = \sum_{s=0}^t |I_s|$. The lemma now follows from the estimate of $\e|I_s|$ in Lemma \ref{lem: expectation of infected}.
\end{proof}

\begin{lemma}\label{lem: nbd hypertree}
    Assume $\lambda>1$. Let $0<\eps<1$ and $\delta = (2\log \lambda)^{-1}(1-\eps)$. We have \[\p ( \text{$B_{\lfloor \delta \log N \rfloor}(i)$ contains a Berge cycle})\le C N^{-\eps}\] for some constant $C>0$ depending only on $(\alpha_p)_{2\le p \le \Delta}$. 
\end{lemma}
\begin{proof}
    We run the exploration process  $(R_t,  I_t, S_t, E_t)_{t\in\mathbb Z_+}$ starting at the vertex $i$. Suppose there is no cycle until step $t-1$.
    Because of (iii) of Proposition \ref{prop: properties of the exploration process},  there are only two  ways to create a potential cycle  at step $t \ge 1$:
    
    \smallskip
    \noindent Case 1. A hyperedge $e\in E_{t}$ contains two distinct vertices $v\neq v'$ in  $I_{t-1}$.
    
    \smallskip
    \noindent Case 2. Two distinct hyperedges $e_1 \ne e_2$ in $E_{t}$ have a common vertex $v$ in $I_t$.

\smallskip
    
\begin{figure}[h]
    \centering
    \begin{minipage}{0.45\textwidth}
        \centering
        \includegraphics[width=0.50\textwidth]{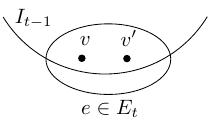} 
    \end{minipage}
    \begin{minipage}{0.45\textwidth}
        \centering
        \includegraphics[width=0.60\textwidth]{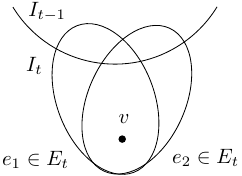} 
    \end{minipage}
 \caption{Left panel depicts Case~1 and right panel depicts Case~2.}
\end{figure}    

    Let $A_t$ and $D_t$ denote the number of hyperedges corresponding to Case 1 and  Case 2 respectively.
    Given $\mathcal F_{t-1}$, $A_t$ is stochastically dominated by $\sum_{2\le p \le \Delta}Y_p$, where $(Y_p)_{2\le p \le \Delta}$ are independent binomial random variables with distribution \begin{equation*}
        Y_p \sim \text{Bin}\Biggl(   |I_{t-1}|^2 {N-2\choose p-2}, \frac{\alpha_p N}{{N\choose p}} \Biggr),
    \end{equation*}
    so by linearity of expectation, we have  
    \begin{equation}\label{eq: stochastic dominance 2}
        \e (A_t|\mathcal F_{t-1}) \le \sum_{p=2}^\Delta \e(Y_p|\mathcal F_{t-1}) =\sum_{p=2}^\Delta  |I_{t-1}|^2 {N-2\choose p-2} \frac{\alpha_p N}{{N\choose p}} =\frac{\lambda |I_{t-1}|^2}{N-1}.
    \end{equation}
    Let us write \[D_t=\sum_{2\le p_1, p_2 \le \Delta} \sum_{|V(e_l)|=p_\ell,\ \ell=1,2} \1_{\{(e_1,e_2)\text{ satisfies Case 2}\}},\] 
    and note that the indicators in the previous display may be dependent if $p_1=p_2$.
    Nonetheless, the number of all such pairs in Case 2 is at most
     \begin{equation*}
        \sum_{p_1,p_2=2}^\Delta N |I_{t-1}|^2 {N-2\choose p_1-2}{N-2\choose p_2-2},
    \end{equation*}
    so we have \begin{equation}\label{eq: stochastic dominance 3}
        \e(D_t|\mathcal F_{t-1}) \le \sum_{p_1,p_2=2}^\Delta N |I_{t-1}|^2 {N-2\choose p_1-2}{N-2\choose p_2-2} \frac{\alpha_{p_1} N}{{N\choose p_1}}\frac{\alpha_{p_2} N}{{N\choose p_2}} =\frac{\lambda^2 N|I_{t-1}|^2}{(N-1)^2}.
    \end{equation}
    By Markov's inequality, \eqref{eq: stochastic dominance 2}, and \eqref{eq: stochastic dominance 3}, we have \[\p(A_t\ge 1 \text{ or } D_t\ge 1|\mathcal F_{t-1})\le \e(A_t|\mathcal F_{t-1})+\e(D_t|\mathcal F_{t-1})\le   \frac{\lambda |I_{t-1}|^2 }{N-1}+ \frac{\lambda^2 N |I_{t-1}|^2}{(N-1)^2}.\] 
    Taking another expectation and summing over $1\le t\le 1+\delta \log N $, we have that for sufficiently large $N$ and $C>0$ depending only on $\lambda$,
    \begin{equation}\label{eq: probability of hypertree}
        \p ( \text{$B_{\lfloor \delta \log N \rfloor}(i) $ contains a cycle}) \le \sum_{1\le t\le 1+\delta \log N}\p(A_t\ge 1 \text{ or } D_t\ge 1)\le \frac{C}{N} \sum_{0\le t \le \delta \log N}\e |I_t|^2.  
    \end{equation}
     On the other hand, from Lemma \ref{lem: expectation of infected}, 
     there is another constant $C$ depending only on $\lambda$ and $\lambda'$ such that \[\sum_{0\le t \le \delta \log N}\e |I_t|^2\le C N^{2\delta \log \lambda}.\]
     Together with \eqref{eq: probability of hypertree}, this finishes the proof.
\end{proof}

Similar to the proof of the disorder chaos in the mixed even $p$-spin short-range model (Section \ref{sec: proof of generalized short-range model}), for any $i,j\in [N]$, we view $\la \sigma_i\sigma_j \ra$ as a measurable function $\phi_{ij}(J,E)$ of the disorder $J$ and the edge set $E$, which is bounded in modulus by $1$. 
Let $\hat\phi_{ij} (n,E) \colonequals \e_J (\phi_{ij}(J,E)h_n(J))$ for any $i,j \in [N]$.


Now, we combine the results obtained so far.
Below $\e$ represents the expectation with respect to two layers of randomness: the disorder and the random graph.
\begin{lemma}\label{lem: upperbound: hypertree}
      Let $\eps>0$ and $\delta = (2\log \lambda)^{-1}(1-\eps)$. Then, for both types of perturbations and any $i,j\in [N]$, we have \[\e \la \sigma_i \sigma_j \tau_i \tau_j \ra_t \le   C N^{-\eps}+  \e e^{-t\min (d(i,j),\delta \log N)}, \]
        where $C>0$ is the constant in Lemma \ref{lem: nbd hypertree}.
\end{lemma}
\begin{proof}
    Define $\e_J$ to be the expectation conditioned on the graph.
   For both types of perturbations and any $i,j\in [N]$, Lemma \ref{lem: semigroup: main} implies \begin{align}
    \e_J \la \sigma_i \sigma_j \tau_i \tau_j \ra_t \le &\sum_{n:|E(n)|< \min(d(i,j),\delta \log N)} e^{-|E(n)|t}\hat \phi_{ij} (n,E)^2 \nonumber \\&\qquad +\sum_{n:|E(n)|\ge \min(d(i,j),\delta \log N)} e^{-|E(n)|t}\hat \phi_{ij} (n,E)^2. \label{eq: upperbound:hypertree:eq1}
    \end{align}
    Note that by Parseval's identity $\sum_{n\in \mathbb Z_+^E} \hat \phi_{ij} (n,E) ^2 = \e_J \phi_{ij} (J,E)^2 \le 1$.
Thus, an application of Proposition~\ref{prop: hypertree lemma} with $r = \delta \log n$ yields
\[ \sum_{n:|E(n)|< \min(d(i,j),\delta \log N)} e^{-|E(n)|t}\hat \phi_{ij} (n,E)^2 \le \1_{\{  \text{$B_{\lfloor\delta \log N \rfloor}(i)$ contains a berge cycle}\} }, \]
since on the complement of the event on the RHS above, $B_{\delta \log n}(i)$ is a hypertree and hence, $\hat \phi_{ij}(n,E)=0$ whenever $|E(n)|<\min( d(i,j), \delta \log N)$. Thanks to Lemma~\ref{lem: nbd hypertree}, the first term in \eqref{eq: upperbound:hypertree:eq1}, after taking an expectation with respect to the randomness of the hypergraph, 
is bounded above by $C N^{-\eps}$. The second term in \eqref{eq: upperbound:hypertree:eq1} can be trivially bounded above by $\e \exp(-t\min(d(i,j),\delta \log N))$ from another use of Parseval's identity.
\end{proof}

\begin{proof}[\bf Proof of Theorem \ref{thm: diluted mixed p spin}]
By symmetry of the diluted mixed $p$-spin model and Lemma \ref{lem: upperbound: hypertree}, we have
\begin{align} 
\e  \la R(\sigma,\tau)^2\ra_t  &=  \frac{1}{N^2} \sum_{1\le i,j\le N} \e \la \sigma_i\sigma_j\tau_i \tau_j\ra_t= \frac{1}{N} \sum_{1\le j\le N} \e \la \sigma_1\sigma_j\tau_1 \tau_j\ra_t \nonumber\\
&\le     C N^{-\eps}+ \frac{1}{N} \sum_{1\le j\le N} \e e^{-t \min(d(1,j),\delta \log N)}\nonumber\\
&\le  C N^{-\eps} + \frac{1}{N}  \e |B_{\lfloor \delta \log N \rfloor}(1) |  + \frac{1}{N} \e \sum_{j: d(1, j) > \delta \log N}  e^{ -t \delta \log N }  \nonumber\\
&\le   C  N^{-\eps} + \frac{1}{N}    \e |B_{\lfloor \delta \log N \rfloor}(1) |  + N^{-\delta t}. \label{eq: upperbound: R^2: 0}
\end{align}
By Lemma~\ref{lem:vol_growth_random_hypg}, the second term of \eqref{eq: upperbound: R^2: 0} is bounded by 
$\e |B_{\lfloor \delta \log N \rfloor}(1)|\le C N^{\delta \log \lambda }$.
Plugging into \eqref{eq: upperbound: R^2: 0}, we obtain \[\e \la R(\sigma,\tau)^2\ra_t \le C N^{-\eps}+ C N^{-1 +\delta \log \lambda}+N^{-\delta t}.\]
Finally, the proof is complete by recalling the definition $\delta= (2\log \lambda)^{-1}(1-\eps)$ to write the previous display as \[\e \la R(\sigma,\tau)^2\ra_t \le 2C N^{-\eps}+N^{-t(1-\eps)/(2\log \lambda)}\] for some constant $C>0$, and then optimizing over $\eps$.
\end{proof}

\section{Proof of Theorem \ref{chaosLevy}}	\label{sec: proof of levy model}
Our proof of Theorem \ref{chaosLevy} is inspired by Chatterjee's semigroup method for the SK model in \cite{Chatterjeebook} that consisted of two major steps. The first is similar to Lemma \ref{add:lem1}$(ii)$ established below that allows one to control $\e \la R(\sigma,\tau)^2\ra_t$ for small $t$ via its behavior for large $t$. The second step aims to control the large $t$ case, which relies on controlling of the derivatives of $\e\la R(\sigma,\tau)^2\ra_t$ for any orders via the Gaussian integration by parts. However, this step is not relevant to our case as the heavy-tailed distribution does not have a finite second moment. To tackle this difficulty, we adapt the convexity argument from Chen-Lam \cite{chen2024universality} instead and the control is made possible through a critical use of the Gaussian hypercontractivity.

We begin our proof by the following lemma that is essentially taken from Chatterjee \cite{Chatterjeebook}.
Let $k\geq 1$ be a fixed integer. 
Let $J$ be a $k$-dimensional standard normal random vector and ${J}^1$ and ${J}^2$ i.i.d. copies of $J.$  
For the sake of symmetry, we set $J^1(t)=e^{-t/2}{J}+\sqrt{1-e^{-t}}{J}^1$ and $J^2(t)=e^{-t/2}{J}+\sqrt{1-e^{-t}}{J}^2$ for $t \ge 0$. Notice that the Gaussian vector $(J^1(t),J^2(t))$ has the same distribution as $(J,J(t))$ in Definition \ref{def: perturbation}.

 \begin{lemma}\label{add:lem1}
		Let $m\ge 1$ and $L_1,\ldots,L_m:\mathbb{R}^k\to \mathbb{R}$ be measurable functions.  Assume that \[
		\max_{1\leq r\leq m}\e |L_r({J})|^{2}<\infty.\]
    Let $\phi(t):=\sum_{r=1}^m\e L_r(J^1(t))L_r(J^2(t))$.	
    For any $0\leq s\leq t $, we have
		\begin{enumerate}
			\item[$(i)$]  $0\leq \phi(t)\leq \phi(s),$
			\item[$(ii)$] 
			$\phi(s)\leq \phi(t)^{\frac{s}{t}}\phi(0)^{1-\frac{s}{t}}.$
		\end{enumerate}
	\end{lemma}
\begin{proof}
It is immediate to see that Lemma \ref{lem: semigroup: main} implies $(i)$. 
Moreover, $(ii)$ follows from H\"older's inequality: 
\begin{align*}
    \phi(s)&=\sum_{r=1}^m \sum_{n\in \mathbb Z_+^{k}}e^{-s|n|} \widehat L_r (n)^2\\
    &\le \Bigl( \sum_{r=1}^m \sum_{n\in \mathbb Z_+^{k}} e^{-s|n|\frac{t}{s}} \widehat L_r (n)^2\Bigr)^{\frac{s}{t}} \Bigl(\sum_{r=1}^m \sum_{n\in \mathbb Z_+^{k}}1\cdot \widehat L_r (n)^2\Bigr)^{1-\frac{s}{t}}
    = \phi(t)^{\frac{s}{t}}\phi(0)^{1-\frac{s}{t}}. 
\end{align*}
\end{proof}

    Define \[H_{N,t}^k(\sigma)= \frac{1}{a_N}\sum_{i,j=1}^N \rho(J_{ij}^k(t))\sigma_i\sigma_j, \quad \sigma \in \{\pm 1\}^N, \, \, k=1,2.\]
	For any $t\ge 0$ and $\lambda>0,$ consider the coupled free energy
	\begin{align*}
		F_N(t,\lambda)&:=\frac{1}{N}\e\ln \sum_{\sigma,\tau}\exp\bigl(\beta H_{N,t}^1(\sigma)+\beta H_{N,t}^2(\tau)+\lambda NR(\sigma,\tau)^2\bigr).
	\end{align*}
Let $\la \cdot\ra_{t,\lambda}$ be the Gibbs expectation associated to this free energy.
We use the following consequence of the Gaussian hypercontractive inequality (see, e.g., page 45 in \cite{Chatterjeebook}): for any function $f:\mathbb R^k \to \mathbb R_+$ such that $f\in L^{1+e^{-t}}(J)$, \[\e f(J^1(t)) f(J^2(t)) \le (\e f(J)^{1+e^{-t}})^{2/(1+e^{-t})}.\] 
	
	\begin{lemma}\label{add:lem2}
		For any $\lambda>0,$ $ t> -\ln (\alpha-1)$, and $1<b<\alpha/(1+e^{-t})$, we have that
		\begin{align*}
			F_N(t,\lambda)-F_N(\infty,\lambda)&\leq \frac{2N\beta^2}{a_N^2}\frac{1}{1-e^{-t}} \bigl(\e |J|^{b/(b-1)}\bigr)^{2(b-1)/b} \bigl(\e |\rho|^{b(1+e^{-t})}(J) \bigr)^{2/((1+e^{-t})b)}.
		\end{align*}
	\end{lemma}
	
	\begin{proof}
	\textbf{Case 1: Smooth $\rho$.}	
 For any $R>0,$ let $\zeta_R:\mathbb{R}\to [0,1]$ be a smooth function satisfying that $\zeta_R\equiv 1$ on $[-R,R]$ and $\zeta_{R}\equiv 0$ on $[-(R+1),R+1]^c$. We further assume that $\sup_{R>0}\|\zeta_R'\|_\infty<\infty.$ Let $\rho_R=\rho\zeta_R$. Let $F_N^R(t,\lambda)$ be the same as $F_N(t,\lambda)$ with the replacement of $\rho$ by $\rho_R$ and let $G_{N,t,\lambda}^R$ be the Gibbs measure associated to $F_N^R(t,\lambda).$ Denote by $(\sigma^1,\sigma^2), (\sigma^1,\tau^1),(\sigma^2,\tau^2)$ the i.i.d. samples of $G_{N,t,\lambda}^R$ and by $\la \cdot\ra_{t,\lambda}^R$ the corresponding Gibbs expectation.  Note that $\rho_R$ is twice differentiable and its derivatives are uniformly bounded. A direct computation gives that
		\begin{align*}
			\partial_tF_N^R(t,\lambda)&=\frac{\beta}{2 N a_N}\sum_{i,j=1}^N \e\Bigl\la \bigl(- e^{-t/2} J_{ij}+ e^{-t}(1-e^{-t})^{-1/2}J_{ij}^1\bigr)\rho_R'(J_{ij}^1(t))\sigma_i\sigma_j\\
			&\qquad\qquad\qquad\qquad+ \bigl(- e^{-t/2} J_{ij}+ e^{-t}(1-e^{-t})^{-1/2}J_{ij}^2\bigr)\rho_R'(J_{ij}^2(t))\tau_i\tau_j\Big\ra_{t,\lambda}^R.
		\end{align*}
		Here, using the Gaussian integration by parts yields
		\begin{align}	
		\nonumber	\e J_{ij}\bigl\la \rho_R'(J_{ij}^1(t))\sigma_i\sigma_j\bigr\ra_{t,\lambda}^R
			&=\frac{e^{-t/2} \beta}{a_N}\e\rho_R'(J_{ij}^1(t))\rho_R'(J_{ij}^1(t))\bigl\la\sigma_i^1\sigma_j^1(\sigma_i^1\sigma_j^1-\sigma_i^2\sigma_j^2)\bigr\ra_{t,\lambda}^R\\
	\nonumber		&+ \frac{e^{-t/2} \beta}{a_N}\e\rho_R'(J_{ij}^1(t))\rho_R'(J_{ij}^2(t))\bigl\la\sigma_i^1\sigma_j^1(\tau_i^1\tau_j^1-\tau_i^2\tau_j^2)\bigr\ra_{t,\lambda}^R\\
	\label{add:eq5}		&+ e^{-t/2}\e\bigl \la \rho_R''(J_{ij}^1(t))\sigma_i\sigma_j\bigr\ra_{t,\lambda}^R .
		\end{align}
		and
		\begin{align}	
		\nonumber	 (1 - e^{-t})^{-1/2} \e J_{ij}^1 \bigl\la \rho_R'(J_{ij}^1(t))\sigma_i\sigma_j\bigr\ra_{t,\lambda}^R
			&=\e\bigl \la \rho_R''(J_{ij}^1(t))\sigma_i\sigma_j\bigr\ra_{t,\lambda}^R\\
		\label{add:eq6}	&+ \frac{\beta}{a_N}\e\rho_R'(J_{ij}^1(t))\rho_R'(J_{ij}^1(t))\bigl\la\sigma_i^1\sigma_j^1(\sigma_i^1\sigma_j^1-\sigma_i^2\sigma_j^2)\bigr\ra_{t,\lambda}^R.
		\end{align}
		Adding these together leads to 
		\begin{align*}
			&\e\Bigl\la \bigl(- e^{-t/2} J_{ij}+ e^{-t}(1-e^{-t})^{-1/2}J_{ij}^1\bigr)\rho'(J_{ij}^1(t))\sigma_i\sigma_j\Bigr\ra_{t,\lambda}^R\\
			&=  \frac{-e^{-t}\beta}{a_N}\e\rho_R'(J_{ij}^1(t))\rho_R'(J_{ij}^2(t))\bigl\la\sigma_i^1\sigma_j^1(\tau_i^1\tau_j^1-\tau_i^2\tau_j^2)\bigr\ra_{t,\lambda}^R.
		\end{align*}
		By symmetry, we then have
		\begin{align}\label{add:eq4}
            \partial_tF_N^R(t,\lambda)&=\frac{-e^{-t}\beta^2}{N a_N^2}\sum_{i,j=1}^N\e\rho_R'(J_{ij}^1(t))\rho_R'(J_{ij}^2(t))\bigl\la\sigma_i^1\sigma_j^1(\tau_i^1\tau_j^1-\tau_i^2\tau_j^2)\bigr\ra_{t,\lambda}^R,
		\end{align}
		which implies that for $b>1$ such that $b(1+e^{-t})<\alpha$ and its conjugate exponent $a=b/(b-1)$,
		\begin{align}
			\nonumber F_N^R(t,\lambda)-F_N^R(\infty,\lambda)&= -\int_t^\infty \partial_sF_N^R(s,\lambda)ds \leq \int_t^\infty\frac{2Ne^{-s}\beta^2}{a_N^2}\e\rho_R'(J^1(s))\rho_R'(J^2(s))ds 
   \\ \nonumber &\leq \frac{2N \beta^2}{a_N^2}\e \rho_R'(J^1(t))\rho_R'(J^2(t)) = \frac{2N\beta^2}{a_N^2}\frac{1}{1-e^{-t}} \e J^1J^2 \rho_R(J^1(t)) \rho_R(J^2(t))
   \\\nonumber   &\le \frac{2N\beta^2}{a_N^2}\frac{1}{1-e^{-t}} (\e |J|^a)^{2/a} (\e |\rho_R|^b(J^1(t)) |\rho_R|^b(J^2(t)))^{1/b}
   \\    &\le \frac{2N\beta^2}{a_N^2}\frac{1}{1-e^{-t}} (\e |J|^a)^{2/a} (\e |\rho_R|^{b(1+e^{-t})}(J) )^{2/((1+e^{-t})b)},			\label{eq: truncated difference}
		\end{align}
	   where the inequalities used Lemma \ref{add:lem1}$(i),$ Gaussian integration by parts, H\"older's inequality, and Gaussian hypercontractive inequality.

  Finally, we argue that our assertion follows by taking $R\to\infty.$ To see this, note that
		\begin{align}\label{eq: truncated F}
			\nonumber&\bigl|F_N(t,\lambda)-F_N^R(t,\lambda)\bigr|\\
			\nonumber&=\frac{1}{N}\Bigl|\e\ln \Bigl\la\exp\frac{\beta}{a_N}\sum_{i,j=1}^N\Bigl(\bigl(\rho(J^1_{ij}(t))-\rho_R(J^1_{ij}(t))\bigr)\sigma_i\sigma_j+\bigl(\rho(J^2_{ij}(t))-\rho_R(J^2_{ij}(t))\bigr)\tau_i\tau_j\Bigr)\Bigr\ra_{t,\lambda}^R\Bigr|\\
			\nonumber&\leq \frac{\beta}{N a_N}\sum_{i,j=1}^N\e\bigl(\bigl|\rho(J^1_{ij}(t))-\rho_R(J^1_{ij}(t))\bigr|+\bigl|\rho(J^2_{ij}(t))-\rho_R(J^2_{ij}(t))\bigr|\bigr)\\
			&=2\beta N a_N^{-1} \e \bigl|\rho(J)-\rho_R(J)\bigr|,
		\end{align}
		where the inequality used the Jensen inequality. Since $\e |\rho (J)|<\infty$ implies
		\begin{align*}
			\e\bigl|\rho(J)-\rho_R(J)\bigr|&\leq \e\bigl[|\rho(J)|;|J|\geq R\bigr]\to 0,\,\,\mbox{as $R\to\infty$,}
		\end{align*}
		we have $\lim_{R\to\infty}F_{N}^R(t,\lambda)=F_N(t,\lambda)$ for any $ t \ge 0.$ 
        Additionally, from \eqref{eq: integrability of heavy tails}, we have  \[\e |\rho|^{b(1+e^{-t})}(J) < \infty \quad  \text{and}\quad \lim_{R\to\infty}\e \bigl(|\rho| ^{b(1+e^{-t})}(J) - |\rho_R|^{b(1+e^{-t})}(J)\bigr)  = 0,\] hence the right-hand side of \eqref{eq: truncated difference} converges as well.

        \textbf{Case 2: General $\rho$.} Define $\rho_{(y)}(x)=\e \rho (x+yJ')$ for $x,y\in \mathbb R$ where $J'$ is standard normal. 
        As $\rho_{(y)}$ is smooth, it belongs to the first case of the proof and enjoys the inequality of the lemma.
       Since $\e |\rho(J)|^p <\infty$, from a standard approximation argument, $\rho_{(y)}$ approximates $\rho(J)$ in $L^p$ for any $1\le p < \alpha$, i.e., $\e |\rho(J)- \rho_{(y)}(J)|^p \to 0$ as $y\downarrow 0$ for any $1\le p<\alpha$.
        Define $F^{(y)}_N(t,\lambda)$ to be $F_N(t,\lambda)$ with $\rho$ replaced by $\rho_{(y)}$.
        The calculations in \eqref{eq: truncated F} show that as $y \downarrow 0$, \[|F^{(y)}_N(t,\lambda)-F_N(t,\lambda)|\le 2\beta  N a_N^{-1} \e |\rho(J)-\rho_{(y)}(J)| \to 0. \] 
        This completes our proof.

  
	\end{proof}
	
	\begin{remark}
		\rm It might look like that \eqref{add:eq4} already holds for $F_N$ following the same argument for $F_N^R$ and one does not need the approximate argument in the proof above. However, in view of \eqref{add:eq5} and \eqref{add:eq6}, it does not seem to be easy to see that
		$$
		\e\rho'(J^1_{ij}(t))\rho'(J^1_{ij}(t))\bigl\la\sigma_i^1\sigma_j^1(\sigma_i^1\sigma_j^1-\sigma_i^2\sigma_j^2)\bigr\ra_{t,\lambda}
		$$
		is finite since it might well be the case $\e\rho'(J)^2=\infty$.
	\end{remark}
	\begin{proof}
		[\bf Proof of Theorem \ref{chaosLevy}] Let $t_0> -\log (\alpha-1)$ and $0<\lambda_0<1/2$ be fixed. Note that \eqref{eq: integrability of heavy tails} ensures that  $\e|\rho(J)|^{1+e^{-t_0}}<\infty.$ From the convexity of $F(t_0,\cdot)$, $F_N(t_0,0)=F_N(0,0),$ and Lemma~\ref{add:lem2}, there exists a constant $K>0$ independent of $N$ such that
		\begin{align*}
			\lambda_0\e\la R(\sigma,\tau)^2\ra_{t_0}&\leq F_N(t_0,\lambda_0)-F_N(t_0,0)\\
			&\leq F_N(\infty,\lambda_0)-F_N(\infty,0)+\frac{KN}{a_N^2}\\
			&=\frac{1}{N}\e\ln\bigl\la \exp \lambda_0 NR(\sigma, \tau)^2\bigr\ra_\infty+\frac{KN}{a_N^2}\\
			&\leq \frac{1}{N}\ln\e\bigl\la \exp \lambda_0 NR(\sigma, \tau)^2\bigr\ra_\infty+\frac{KN}{a_N^2}.
		\end{align*}
		Observe that since $(\rho(J_{ij}^1))_{i,j\in[N]}$ and $(\rho(J_{ij}^2))_{i,j\in[N]}$ are i.i.d. and symmetric, we see that with respect to the measure $\e\la \cdot\ra_\infty$, $NR(\sigma,\tau)$ is equal to $B_1+\cdots+B_N$ in distribution for i.i.d. Rademacher $(1/2)$ random variables $B_1,\ldots,B_N.$ Consequently,
		\begin{align*}
			\e\bigl\la \exp \lambda_0 NR(\sigma,\tau)^2\bigr\ra_\infty &=\e \exp \frac{\lambda_0}{N}(B_1+\cdots+B_N)^2\leq \frac{1}{\sqrt{1-2\lambda_0}}.
		\end{align*}
  Observe that \[\la R(\sigma,\tau)^2\ra_t = \frac{1}{N^2}\sum_{i,j=1}^N \phi_{ij}(J^1(t))\phi_{ij}(J^2(t)),\]
    where for each $(i,j)\in [N]^2$, we define the bounded
    function $\phi_{ij}:\mathbb R^{N^2} \to \mathbb R$ by \[\phi_{ij}(x)= \Bigl(\sum_{\sigma}\exp\Bigl(\frac{\beta}{a_N}\sum_{k,l=1}^N\rho (x)\sigma_k\sigma_l\Bigr)\Bigr)^{-1} \sum_{\sigma}\sigma_i\sigma_j \exp\Bigl(\frac{\beta}{a_N}\sum_{k,l=1}^N\rho (x)\sigma_k\sigma_l\Bigr), \quad x\in \mathbb R^{N^2}.\]
		Hence, there exists a constant $K'$ such that for any $t\ge t_0$ and $N\geq 1,$
		\begin{align*}
			\e\la R(\sigma,\tau)^2\ra_{t} \leq \e\la R(\sigma,\tau)^2\ra_{t_0}&\leq \frac{1}{\lambda_0}\Bigl(\frac{1}{N}\ln \frac{1}{\sqrt{1-2\lambda_0}}+\frac{KN}{a_N^2}\Bigr)\leq \frac{K'N}{a_N^2},
		\end{align*}
	where the first inequality used Lemma \ref{add:lem1}$(i)$. 
		Note that $\e\la R(\sigma,\tau)^2\ra_{0}\leq 1$. From Lemma \ref{add:lem1}$(ii),$ for any $0\le t \le t_0$ and $N\geq 1,$
		\begin{align*}
			\e\la R(\sigma,\tau)^2\ra_{t}&\leq (\e\la R(\sigma,\tau)^2\ra_{t_0})^{t/t_0}\leq K'' \Bigl(\frac{N}{a_N^2}\Bigr)^{ t/ t_0},
		\end{align*}
		where $K''\geq K'$ is a constant independent of $t\in [0,t_0]$ and $N\geq 1.$
		Putting these together, we conclude that
		\begin{align*}
			\e\la R(\sigma,\tau)^2\ra_{t}&\leq K''\Bigl(\frac{N}{a_N^2}\Bigr)^{\min( 1, t/t_0)} .
		\end{align*}
    The proof is now complete in view of the estimate for $a_N$ obtained from \eqref{potter} and \eqref{eq: a_N}.
	\end{proof}

\begin{appendix}
\section{Proof of item 2 in Remark \ref{rem:Theorem_1.3}}
Throughout this section, set $\phi_{ij}(J)=\la \sigma_i\sigma_j\ra$ for $1\le i,j\le N$. 
The case for the discrete perturbation is quite simple.

\begin{proposition}
Under the discrete perturbation, if $t\le |E|^{-1}$, then \[\e \la R(\sigma,\tau)^2\ra_t  \ge  e^{-1}\e\la R(\sigma^1,\sigma^2)^2\ra.\]
\end{proposition}
\begin{proof}
Equation \eqref{eq: semigroup: discrete} and Parseval's identity immediately shows \begin{align*}
    \e \la R(\sigma,\tau)^2\ra_t  = \frac{1}{N^2} \sum_{1\le i,j\le N} \e \la \sigma_i\sigma_j\tau_i \tau_j\ra_t &= \frac{1}{N^2}\sum_{1\le i,j\le N } \sum_{n\in\Z_+^E} e^{-t |E(n)|} \hat \phi_{ij}(n)^2 \\&\ge e^{-1}\frac{1}{N^2}\sum_{1\le i,j\le N } \e \la \sigma_i \sigma_j\ra ^2 = e^{-1}\e\la R(\sigma^1,\sigma^2)^2\ra.
\end{align*}
\end{proof}

Now, let us consider the continuous perturbation.
\begin{proposition} \label{prop: continuous perturbation: lower bound 1}
    Suppose that for each $2\le p\le \Delta$, $\rho_p$ is thrice differentiable and \[\max( \|\rho'_p\|_{\infty},\| \rho_p''\|_{\infty},\| \rho_p'''\|_{\infty}) <\infty .\]
    Under continuous perturbation, if $t\le |E|^{-5/2-\eps}$ for some $\eps>0$, then $ \e \la R(\sigma,\tau)^2\ra_{ t}\ge  c\e\la R(\sigma^1,\sigma^2)^2\ra $ for some constant $c>0,$ independent of $\beta$.
\end{proposition}
In the special case of the Gaussian disorder, we have another lower bound on the perturbed overlap, which has a better dependence on $|E|$ but gets worse as $\beta\to\infty.$
\begin{proposition}\label{prop: continuous perturbation: lower bound 2}
    Suppose $\rho_p(x)=x$ for $x\in\mathbb R$ and any $2\le p\le \Delta$. Under the continuous perturbation, for any $t\ge0$, \[ \e \la R(\sigma,\tau)^2\ra_t \ge  \e \la R(\sigma,\tau)^2\ra_0 - 6\sqrt{t}\sqrt{\beta} |E|^{3/4}.\]
    In particular, $\e \la R(\sigma,\tau)^2\ra_t >2^{-1}\e\la R(\sigma^1,\sigma^2)^2\ra$ holds if 
    \[t\le   12^{-2}\beta^{-1}|E|^{-3/2} (\e \la R(\sigma^1,\sigma^2)^2\ra)^2.\]
\end{proposition}
In the rest of the appendix, we provide the proofs of Propositions \ref{prop: continuous perturbation: lower bound 1} and \ref{prop: continuous perturbation: lower bound 2}.
First, we make the following observation that relates the overlap under the continuous perturbation to the one under the discrete perturbation.
\begin{lemma}\label{prop: discrete and continuous}
 Let $\phi:\mathbb R^E \to \mathbb R$ be such that $\e \phi^2(J)\le 1$.
    For $t\ge 0$, set \begin{equation*}
	f(t)=\sum_{n\in \mathbb{Z}_+^E}e^{-t|E(n)|}\hat \phi(n)^2, \quad g(t)=\sum_{n\in \mathbb{Z}_+^E}e^{-t|n|}\hat \phi(n)^2.
\end{equation*}
    For any $0<\gamma<2^{-1}$ and $m\geq 1$ satisfying $2\gamma m<1,$ we have that 
	\begin{align*}
		f(t)&\leq g(\gamma t)+\frac{1}{((\gamma m)^{-1}-1)(\gamma^{-1}-2)}\sum_{a\in E}\e \bigl\|\nabla \partial_{x_a}^2 \phi(J)\bigr\|^2+e^{-tm}.
	\end{align*}
\end{lemma}
\begin{proof}
Let $0<\gamma<1.$ Write
\begin{align*}
    f(t)&=\sum_{|E(n)|\geq \gamma |n|}e^{-t|E(n)|}\hat \phi(n)^2+\sum_{|E(n)|<\gamma |n|}e^{-t|E(n)|}\hat \phi(n)^2\\
    &\leq \sum_{|E(n)|\geq \gamma |n|}e^{-\gamma t|n|}\hat \phi(n)^2+\sum_{|E(n)|<\gamma |n|}e^{-t|E(n)|}\hat \phi(n)^2\\
    &\leq g(\gamma t)+\sum_{|E(n)|<\gamma |n|}e^{-t|E(n)|}\hat \phi(n)^2.
\end{align*}
Here, for any $m\geq 1,$ we further write
\begin{align*}
    \sum_{|E(n)|<\gamma |n|}e^{-t|E(n)|}\hat \phi(n)^2&=\sum_{|E(n)|<\gamma |n|,|E(n)|\leq m}e^{-t|E(n)|}\hat \phi(n)^2+\sum_{|E(n)|<\gamma |n|,|E(n)|>m}e^{-t|E(n)|}\hat \phi(n)^2\\
    &\leq \sum_{|E(n)|<\gamma |n|,|E(n)|\leq m}e^{-t|E(n)|}\hat \phi(n)^2+e^{-tm}\sum_{|E(n)|<\gamma |n|,|E(n)|>m}\hat \phi(n)^2\\
    &\leq \sum_{|E(n)|<\gamma |n|,|E(n)|\leq m}e^{-t|E(n)|}\hat \phi(n)^2+e^{-tm},
\end{align*}
where the last inequality used $\e \phi(J)^2=\sum_{n}\hat \phi(n)^2\leq 1.$
To control the sum in the last inequality, we write
\begin{align*}
	\sum_{|E(n)|<\gamma |n|,|E(n)|\leq m}e^{-t|E(n)|}\hat \phi(n)^2&=
	\sum_{1\leq |E(n)|<\gamma |n|,|E(n)|\leq m}e^{-t|E(n)|}\hat \phi(n)^2\\
	&\leq  \sum_{a\in E}\sum_{|n|\geq\gamma^{-1},|E(n)|\leq m,n_a\geq \gamma^{-1}/m }\hat \phi(n)^2\\
	&\leq  \sum_{a\in E}\sum_{|n|\geq \gamma^{-1},n_a\geq\gamma^{-1}/m }\hat \phi(n)^2,
\end{align*}
where the first inequality holds due to the pigeonhole principle.
Now, for any $n$ satisfying $|n|\geq \gamma^{-1}$ and $n_a\geq (\gamma m)^{-1}$, we can write
\begin{align*}
	\hat \phi(n)
	&=\e \Bigl(\prod_{b\neq a}h_{n_b}(J_b)\Bigr)h_{n_a}(J_a)\phi(J)\\
	&=\frac{1}{\sqrt{n_a(n_a-1)}}\e \Bigl(\prod_{b\neq a}h_{n_b}(J_b)\Bigr)h_{n_a-2}(J_a)\partial_{x_a}^2\phi(J)
\end{align*}
so that
\begin{align*}
	\hat \phi(n)^2
	&\leq \frac{1}{n_a-1}\Bigl(\e \Bigl(\prod_{b\neq a}h_{n_b}(J_b)\Bigr)h_{n_a-2}(J_a)\partial_{x_a}^2\phi(J)\Bigr)^2.
\end{align*}
Consequently,
\begin{align*}
	\sum_{|n|\geq \gamma^{-1},n_a\geq \gamma^{-1}/m }\hat \phi(n)^2&\leq\frac{1}{(\gamma m)^{-1}-1} \sum_{\substack{n_a-2+(|n|-n_a)\geq\gamma^{-1}-2,
    \\n_a-2\geq\gamma^{-1}/m-2 }}\Bigl(\e \Bigl(\prod_{b\neq a}h_{n_b}(J_b)\Bigr)h_{n_a-2}(J_a)\partial_{x_a}^2\phi(J)\Bigr)^2\\
	&= \frac{1}{(\gamma m)^{-1}-1}\sum_{|n|\geq\gamma^{-1}-2,n_a\geq\gamma^{-1}/m-2 }\widehat{ \partial_{x_a}^2f}(n)^2\\
	&\leq \frac{1}{((\gamma m)^{-1}-1)(\gamma^{-1}-2)}\|\nabla \partial_{x_a}^2f\|^2,
\end{align*}
where the last inequality used Lemma \ref{lem: improved Poincare}. Putting these bounds together yields our assertion.
\end{proof}

\begin{lemma}[Equation (6.3), \cite{chatterjee2014superconcentration}]\label{lem: improved Poincare}
	For any differentiable $\phi:\mathbb{R}^E\to\mathbb{R}$ with $\e \phi(J)^2<\infty$ and $\e \|\nabla \phi(J)\|^2<\infty$ and $m\geq 1,$ we have that
	\begin{align*}
		\sum_{|n|\geq m}\hat \phi(n)^2\leq \frac{\e \|\nabla \phi(J)\|^2}{m}.
	\end{align*}
\end{lemma}
\begin{proof}[Proof of Proposition \ref{prop: continuous perturbation: lower bound 1}]
    Notice that $ \max_{2\le p\le \Delta}\max( \|\rho'_p\|_{\infty},\| \rho_p''\|_{\infty}\| \rho_p'''\|_{\infty}) <\infty $ implies that \[\sum_{a\in E}\e \bigl\|\nabla \partial_{x_a}^2 \phi_{ij}(J)\bigr\|^2\le K |E|^2\] for some constant $K>0$.
Let $\delta= \eps/2$,  $m= |E|^{1+2\delta}$, and $\gamma = t |E|$.
Notice that $t\ll |E|^{-5/2-\delta}$.
Moreover, $\gamma^2m\ll |E|^{-2}$, $tm\gamma^{-1}\gg1$, and, therefore,  Lemma \ref{prop: discrete and continuous} implies that \begin{align*}
    \e \la R(\sigma,\tau)^2\ra_{ t}  = \frac{1}{N^2} \sum_{1\le i,j\le N} \e \la \sigma_i\sigma_j\tau_i \tau_j\ra_{ t} &= \frac{1}{N^2}\sum_{1\le i,j\le N } \sum_{n\in\Z_+^E} e^{- t |n|} \hat\phi_{ij}(n) ^2 \\&\ge \frac{1}{N^2}\sum_{1\le i,j\le N } \bigl( e^{-1} \e \la \sigma_i\sigma_j\ra^2 - o(1)\bigr)= c\e\la R(\sigma^1,\sigma^2)^2\ra 
\end{align*} for some constant $c>0$.
\end{proof}

\begin{proof}[Proof of Proposition \ref{prop: continuous perturbation: lower bound 2}]
    For the sake of symmetry,
we set $J^1(t)=e^{-t/2}{J}+\sqrt{1-e^{-t}}{J}^1$ and $J^2(t)=e^{-t/2}{J}+\sqrt{1-e^{-t}}{J}^2$ for $t \ge 0$, where $J, J^1, J^2$ are i.i.d.\ standard Gaussian random variables. 
Notice that the Hamiltonians satisfy \[\partial_t H^i_t= -2^{-1}e^{-t/2} H_{J} + 2^{-1}e^{-t}(1-e^{-t})^{-1/2} H_{J^i}, \quad i=1,2.\]
We can compute the derivative \begin{align*}
    \beta^{-1}\frac{d}{dt}\e \la R(\sigma,\tau)^2\ra_t &=   \e \bigl\la  R(\sigma^1,\tau^1)^2 \bigl(\partial_t H^1_t(\sigma^1)+\partial_t H^2_t (\tau^1)- \partial_t H^1_t(\sigma^2)-\partial_t H^2_t (\tau^2)  \bigr)\bigr\ra_t
    \\&= \e \bigl\la  R(\sigma^1,\tau^1)^2 \bigl(\partial_t H^1_t(\sigma^1)- \partial_t H^1_t(\sigma^2) +\partial_t H^2_t (\tau^1)-\partial_t H^2_t (\tau^2)  \bigr)\bigr\ra_t
    \\&= \e \bigl\la  R(\sigma^1,\tau^1)^2 \bigl( -2^{-1}e^{-t} (H_J(\sigma^1)-H_J(\sigma^2)) \\&\qquad \qquad +2^{-1}e^{-t}(1-e^{-t})^{-1/2}(H_{J^1}(\sigma^1)-H_{J^1}(\sigma^2))\bigr)\bigr\ra_t
    \\&\quad + \e \bigl\la  R(\sigma^1,\tau^1)^2 \bigl( -2^{-1}e^{-t} (H_J(\tau^1)-H_J(\tau^2)) 
    \\&\qquad \qquad +2^{-1}e^{-t}(1-e^{-t})^{-1/2}(H_{J^2}(\tau^1)-H_{J^2}(\tau^2))\bigr)\bigr\ra_t.
\end{align*}
Theorem 4 in \cite{CP14} implies that \[\e \bigl\la |H_{J}(\sigma^1)- H_{J}(\sigma^2)|^2 \bigr\ra_t \le \frac{2\sqrt{2} \cdot |E|^{3/2}}{\beta},\] and similar results for $J$ replaced with $J^1$ and $J^2$, or $\sigma $ replaced with $\tau.$
Therefore, taking the absolute value of the derivative and using Jensen's inequality, \begin{align*}
    \beta^{-1}\frac{d}{dt}\e \la R(\sigma,\tau)^2\ra_t &\le e^{-t}(1+(1-e^{-t})^{-1/2}) \bigl(\e \bigl\la |H_J(\sigma^1)-H_J(\sigma^2)|^2 \bigr\ra_t \bigr)^{1/2}
    \\&\le 2^{3/4} e^{-t}(1+(1-e^{-t})^{-1/2}) \frac{|E|^{3/4}}{\sqrt{\beta}}.
\end{align*}
By the fundamental theorem of calculus,
\begin{align*}
    \e \la R(\sigma,\tau)^2\ra_t &\ge \e \la R(\sigma,\tau)^2\ra_0 -\int_0^t 2^{3/4} e^{-u}(1+(1-e^{-u})^{-1/2}) \sqrt{\beta}|E|^{3/4} du
    \\ & =  \e \la R(\sigma,\tau)^2\ra_0 - 2^{3/4}\bigl(1-e^{-t}+2\sqrt{1-e^{-t}}\bigr)    \sqrt{\beta}|E|^{3/4}
     \\ & \ge \e \la R(\sigma,\tau)^2\ra_0 - 6\sqrt{t} \sqrt{\beta}|E|^{3/4}.
\end{align*}
This implies that if $t$ is small enough such that the second term is less than or equal to $\frac{1}{2}\e \la R(\sigma,\tau)^2\ra_0$, i.e., \[t \le 12^{-2}\beta^{-1}|E|^{-3/2} (\e \la R(\sigma^1,\sigma^2)^2\ra)^2, \]  then $\e \la R(\sigma,\tau)^2\ra_t >2^{-1}\e\la R(\sigma,\tau)^2\ra_0.$
\end{proof}

\end{appendix}

\bibliographystyle{acm}

{\footnotesize\bibliography{references}}

\end{document}